\documentclass[11pt,a4paper]{article}
\usepackage{ifthen,latexsym,amssymb,amsmath,bbm,amsthm}%,fixmath}
\usepackage[shortlabels]{enumitem}

\usepackage[nobysame,initials]{amsrefs}

\usepackage{hyperref}
\usepackage{tikz}
\usepackage{dsfont}
\usepackage{graphicx}
\usepackage{anyfontsize}
\usepackage{array}
\usepackage{lineno}

\usepackage{graphicx} % Required for inserting images
\usepackage{amsmath, amsthm, amssymb, mathrsfs}
\usepackage{hyperref}

\newtheorem{thm}{Theorem}[section]
\newtheorem{prop}[thm]{Proposition}
\newtheorem{claim}[thm]{Claim}
\newtheorem{lemma}[thm]{Lemma}
\newtheorem{cor}[thm]{Corollary}
\newtheorem{conj}[thm]{Conjecture}

\newcommand{\C}[1]{\mathcal{#1}}
\renewcommand{\O}[1]{\overline{#1}}
\newcommand{\GR}[1]{\mathrm{GR}(#1)}
\newcommand{\cG}[2]{{\mathcal G}^{(#1)}_{#2}}
\def\ex{\mathrm{ex}}
\newcommand{\p}[2]{P_{#1}(#2)}
\newcommand{\pp}[3]{P_{\,\O{#1}#2}(#3)}
\newcommand{\CI}{C}
\newcommand{\I}[1]{{\mathbbm #1}}
\newcommand{\OutIn}[2]{$\ifthenelse{\equal{#1}{}}{}{\O #1}#2$}

\newcommand{\OneClaim}{$1$-claim}
\newcommand{\OMerge}[2]{\ifthenelse{\equal{#1}{}}{(#2)}{(#1|#2)}}
\newcommand{\Merge}[2]{\ifthenelse{\equal{#1}{}}{#2}{#1|#2}}
\newcommand{\cluster}{cluster}
\newcommand{\OneClusters}{\mathcal{M}_{\Merge{}{1}}}
\newcommand{\OneCluster}{$\Merge{}{1}$-\cluster}

\newcommand{\TwoCluster}{$\Merge{}{2}$-\cluster}

\newcommand{\hide}[1]{}
\newcommand{\op}[1]{\textcolor{blue}{OP: #1}}
\newcommand{\sm}[1]{\textcolor{orange}{SS: #1}}

\title{\bf\Large On the quadratic 8-edge case of\\  the Brown--Erd\H{o}s--S\'os problem}

\author{Oleg Pikhurko and Shumin Sun\\
Mathematics Institute and DIMAP\\
University of Warwick\\
Coventry CV4 7AL, UK}

\date{}

\begin{document}

\maketitle
\begin{abstract}
    Let $f^{(r)}(n;s,k)$ be the maximum number of edges in an $n$-vertex $r$-uniform hypergraph containing no  $k$ edges on at most $s$ vertices. Brown, Erd\H{o}s and S\'os  conjectured in 1973 that the limit $\lim_{n\rightarrow \infty}n^{-2}f^{(3)}(n;k+2,k)$ exists for all $k$. Recently, Delcourt and Postle settled the conjecture and their approach was generalised by Shangguan to every uniformity $r\ge 4$: the limit $\lim_{n\rightarrow \infty}n^{-2}f^{(r)}(n;rk-2k+2,k)$ exists for all $r\ge 3$ and $k\ge 2$. 
    
    The value of the limit is currently known for $k\in \{2,3,4,5,6,7\}$ due to various results authored by Glock, Joos, Kim, K\"{u}hn, Lichev, Pikhurko, R\"odl and Sun. In this paper we consider the case $k=8$, determining the value of the limit for each $r\ge 4$ and presenting a lower bound for $r=3$ that we conjecture to be sharp.

\end{abstract}
\section{Introduction}

 For an integer $r\geq 2$, an \textit{$r$-uniform hypergraph} (in short, \textit{$r$-graph}) $H$ consists of a vertex set $V(H)$ and an edge set $E(H)\subseteq \binom{V(H)}{r}$, that is, every edge is an $r$-element subset of $V(H)$. Given a family $\C F$ of $r$-graphs, the \emph{Tur\'an number} of $\C F$, denoted by $\ex(n;\C F)$, is defined as the maximum number of edges in an $n$-vertex $r$-graph containing no element of $\C F$ as a subgraph. 

In this paper, we focus on the family $\C F^{(r)}(s,k)$, consisting of all $r$-graphs with $k$ edges and at most $s$ vertices. Brown, Erd\H{o}s and S\'os~\cite{71BES} initiated the systematic investigation of the function
$$f^{(r)}(n;s,k):=\ex(n;\C F^{(r)}(s,k)).$$
They showed that
$$
\Omega(n^{(rk-s)/(k-1)})=f^{(r)}(n;s,k)=O(n^{\lceil (rk-s)/(k-1)\rceil}).
$$
If  $t:=(rk-s)/(k-1)$ is an interger, i.e. $s=rk-tk+t$, then  $f^{(r)}(n;rk-tk+t,k)=\Theta (n^t)$. In the sequel, we are mainly interested in the case when $t=2$; thus $s=rk-2k+2$ and the magnitude of the function is $\Theta (n^2)$. A natural question here is whether the limit
$$\pi(r,k):=\lim_{n\to \infty}n^{-2}f^{(r)}(n;rk-2k+2,k)$$ 
exists, which was originally conjectured by Brown, Erd\H{o}s and S\'os for $r=3$.

In their initial paper~\cite{71BES}, Brown, Erd\H{o}s and S\'os confirmed the conjecture for $k=2$ by observing that $\pi(3,2)=1/6$. 
Many years later, Glock~\cite{19G} solved the case $k=3$ and showed that $\pi(3,3)=1/5$. In a recent work, Glock, Joos, Kim, K\"{u}hn, Lichev and Pikhurko~\cite{GlockJoosKimKuhnLichevPikhurko24} proved the case $k=4$ by showing that $\pi(3,4)=7/36$. Building on their work, Delcourt and Postle~\cite{DelcourtPostle24} finally resolved the Brown--Erd\H{o}s--S\'os conjecture, namely, $\pi(3,k)$ exists for any $k\ge 2$, without determining its value. 

For more general case $\pi(r,k)$ with uniformity $r\ge 4$, the existence of the limit  (without explicit value) was shortly confirmed by Shangguan~\cite{23sh}, following the approach of Delcourt and Postle. The natural remaining question is to determine the limits. Recently, a range of results regarding this direction has been established. Apart from values for $r=3$ mentioned above, the celebrated work of R\"odl~\cite{Rodl85} on the existence of approximate Steiner systems implies that $\pi(r,2)=\frac{1}{r^2-r}$ for every $r\ge 3$. Moreover, Glock, Joos, Kim, K\"{u}hn, Lichev and Pikhurko~\cite{GlockJoosKimKuhnLichevPikhurko24} proved that for every $r\ge 3$,
$$\pi(r,3)=\frac{1}{r^2-r-1}\quad \text{and}\quad \pi(r,4)=\frac{1}{r^2-r}.$$
Very recently, Glock, Kim, Lichev, Pikhurko and Sun~\cite{GKLPS24} obtained the limits for $k\in \{5,7\}$, which is same as $k=3$.
$$\pi(r,5)=\pi(r,7)=\frac{1}{r^2-r-1}\quad \text{for every $r\ge 3$.}$$
They also resolved the case $k=6$. Curiously, in this case, the value behaves differently when $r=3$ and $r\ge 4$ as follows:
$$\pi(3,6)=\frac{61}{330}\quad \text{and\quad $\pi(r,6)=\frac{1}{r^2-r}$ for every $r\ge 4$.}$$
Meanwhile, Letzter and Sgueglia~\cite{25LS} provided the exact value
\begin{equation}\label{eq:evenlimit}
\pi(r,k)=\frac{1}{r^2-r}
\end{equation}
for even integer $k$ and $r\ge r_0(k)$ sufficiently large. In their paper, they asked for the smallest $r$ such that (\ref{eq:evenlimit}) holds.

In this paper, we determine the limit for $k=8$ and $r\ge 4$ as follows.

\begin{thm}\label{thm:main}
    For every $r\ge 4$, we have $\pi(r,8)=\frac{1}{r^2-r}.$
\end{thm}

Moreover, we provide a lower bound for $r=3$, which implies that $r=4$ is the smallest uniformity such that $\pi(r,8)=\frac{1}{r^2-r}$.

\begin{thm}\label{thm:lb}
    $\pi(3,8)\ge \frac{3}{16}$.
\end{thm}

We conjecture that this lower bound is sharp.

\begin{conj}
    $\pi(3,8)=\frac{3}{16}$.
\end{conj}

\noindent \textbf{A connection to generalised Ramsey numbers.}
\hide{In a recent work, Bennett, Cushman and Dudek~\cite{BennettCushmanDudek25} found a connection between generalised Ramsey numbers and the Brown--Erd\H{o}s--S\'os function for $4$-graphs.
Our result that $\pi(4,8)=1/12$ directly gives the asymptotic value for a generalized Ramsey numbers. Before presenting the formal statement, we first introduce some necessary definitions and notation. 
}

In a recent work, Bennett, Cushman and Dudek~\cite{BennettCushmanDudek25} found a connection between the Brown--Erd\H{o}s--S\'os function for $4$-graphs and generalised Ramsey numbers, that were introduced by Erd\H{o}s and Shelah~\cite{ES75} and were first systematically studied by Erd\H{o}s and Gy\'arf\'as~\cite{EG97}.

For integers $p, q$ such that $p \geq 3$ and $2\le q\le\tbinom{p}{2}$, a \emph{$(p, q)$-colouring} of $K_n$
is a colouring of the edges of $K_n$ such that every clique of size $p$ receices at least $q$ colours. The \emph{generalised Ramsey number} $\GR{n,p,q}$ is the minimum number of colours such that $K_n$ has a $(p, q)$-colouring. 

\hide{
Erd\H{o}s and Gy\'arf\'as~\cite{EG97} proved that for  arbitrary $p\geq 3$ and $q_{\rm{lin}} := 
%q_{\rm{lin}}(p) = 
\tbinom{p}{2}-p+3$, 
\[\GR{n,p,q_{\rm{lin}}} = \Omega(n)\quad \text{and}\quad \GR{n,p,q_{\rm{lin}}-1} = o(n).\]
Similarly, they showed that for $q_{\rm{quad}} := 
%q_{\rm{quad}}(p) = 
\tbinom{p}{2}-\lfloor p/2\rfloor+2$,
\[\GR{n,p,q_{\rm{quad}}} = \Omega(n^2)\quad \text{and}\quad \GR{n,p,q_{\rm{quad}}-1} = o(n^2).\]
Thus, for fixed $p$, we say $q_{\rm{lin}}$ the \emph{linear threshold} and $q_{\rm{quad}}$ the \emph{quadratic threshold} of generalised Ramsey number.
}

Erd\H{o}s and Gy\'arf\'as~\cite{EG97} proved among other results that for  arbitrary $p\geq 3$ and $q_{\rm{quad}} := 
%q_{\rm{quad}}(p) = 
\tbinom{p}{2}-\lfloor p/2\rfloor+2$, it holds that
\[\GR{n,p,q_{\rm{quad}}} = \Omega\left(n^2\right)\quad \text{and}\quad \GR{n,p,q_{\rm{quad}}-1} = o\left(n^2\right).\]
Thus $q_{\rm{quad}}$ is the threshold for quadratic growth.

Bennett, Cushman and Dudek~\cite{BennettCushmanDudek25} showed the following connection between generalised Ramsey numbers and the Brown--Erd\H{o}s--S\'os function.

\begin{thm}[\cite{BennettCushmanDudek25}*{Theorem 3}]
\label{thm:quadratic}
For all even $p\geq 6$, we have
\[\lim_{n\to \infty} \frac{\GR{n,p,q_{\rm{quad}}}}{n^2} = \frac{1}{2} - \pi \left(4,\frac{p}{2}-1\right).\]
In particular, the limit on the left exists by~\cite{23sh}. %%\qed
%Furthermore, there exist asymptotically optimal $(p,q_{\rm{quad}})$-colourings that use no colour more than twice.
\end{thm}

From Theorem~\ref{thm:main}, together with Theorem~\ref{thm:quadratic}, we directly obtain the following asymptotic value of $\GR{n,18,146}$.

\begin{thm}
$$ 
\lim_{n\to \infty} \frac{\GR{n,18,146}}{n^2} = \frac{5}{12}.
$$
\end{thm}

\noindent\textbf{Organisation.} The remainder of this paper is organised as follows. We introduce some necessary definitions and notation in Section~\ref{se:prelim}. Section~\ref{se:lower} provides the proofs of lower bounds in Theorem~\ref{thm:main} and~\ref{thm:lb}. The proof of upper bound in Theorem~\ref{thm:main} can be found in Section~\ref{se:upper}. 

\section{Preliminaries}\label{se:prelim}

Throughout the paper, we shall use the following notation and terminology.
For integers $m$ and~$n$, we denote by $[n]$ the set $\{1,\dots,n\}$ and by $[m,n]$ the set $[n]\setminus [m-1]=\{m,\dots,n\}$. For a set $X$, we define ${X\choose m}:=\{Y\subseteq X: |Y|=m\}$ to be the family of all $m$-subsets of~$X$. For simplicity, we often denote the unordered pair $\{x,y\}$ (resp.\ triple $\{x,y,z\}$)  by $xy$ (resp.\ $xyz$). 

%We will often write an unordered pair $\{x,y\}$ (resp.\ triple $\{x,y,z\}$) as~$xy$ (resp.\ as~$xyz$).  
%Moreover, for three real numbers $a$, $b$ and $c\geq 0$, we write $a = b\pm c$ to say that $a\in [b-c, b+c]$. Also, we write $a\gg b>0$ to mean that $b$ is a sufficiently small positive real depending on $a$.

We will often identify an $r$-graph $G$ with its edge set. In particular, if we specify only the edge set $E(G)$ then the vertex set is assumed to be the union of these edges, that is, $V(G)=\bigcup_{e\in E(G)} e$. We let $|G|$ be the number of edges of $G$ and $v(G)$ be the number of vertices of~$G$.
For two $r$-graphs $G$ and $H$, we define their \emph{union} $G\cup H$ by $E(G\cup H):=E(G)\cup E(H)$, and their \emph{difference} $G\setminus H$ by $E(G\setminus H):=E(G)\setminus E(H)$. 
By a \emph{graph}, we will mean a 2-graph.

A \emph{diamond} is an $r$-graph consisting of two edges sharing exactly two vertices.
For positive integers $s$ and $k$, an \emph{$(s,k)$-configuration} is an $r$-graph with $k$ edges and at most $s$ vertices, that is, an element of $\C F^{(r)}(s,k)$. In particular, if $s=rk-2k+2$, we simply omit $s$ and refer to it as a \emph{$k$-configuration}. Moreover, if $s=rk-2k+1$, we refer to it as \emph{$k^-$-configuration}. We say that an $r$-graph is \emph{$k$-free} (resp.\ \emph{$k^-$-free}) if it contains no $k$-configuration (resp.\ $k^-$-configuration). Let $\cG{r}{k}$ denote the family of all $k$-configurations and all $\ell^-$-configurations with $\ell\in [2,k-1]$, namely, 
$$\cG{r}{k}:=\C F^{(r)}(rk-2k+2,k) \cup\left(\bigcup_{\ell=2}^{k-1}\C F^{(r)}(r\ell-2\ell+1,\ell)\right).$$

Note that $\cG{r}{k}$ not only contains $k$-configurations, which is primary topic of this paper, but also includes ``denser'' $r$-graphs of smaller sizes. In the following sections, we will see that this family is closely related to the lower and upper bounds on $\pi(r,k)$. 

We use the following definitions introduced in~\cite{GKLPS24}. For an $r$-graph $G$, a pair $xy$ of distinct vertices (not necessarily in ${V(G)\choose 2}$) and $A\subseteq\I N\cup\{0\}$, we say that $G$ \emph{$A$-claims} the pair $xy$ if, for every $i\in A$, there are $i$ distinct edges $e_1,\dots,e_i\in E(G)$ such that $|\{x,y\}\cup (\bigcup_{j=1}^i e_j)|\leq ri-2i+2$. In particular, if $xy\in {V(G)\choose 2}$, this is equivalent to the existence,  for every $i\in A$,  of an $i$-configuration $J\subseteq G$ such that $\{x,y\}\subseteq V(J)$. Let $\p{A}{G}$ be the set of all pairs in ${V(G)\choose 2}$ that are $A$-claimed by~$G$.  If $A=\{i\}$ is a singleton, we simply write \emph{$i$-claims} (resp.\ $\p{i}{G}$) instead of $\{i\}$-claims (resp.\ $\p{\{i\}}{G}$).
%By definition, any $r$-graph  $0$-claims any pair  (which will be notationally convenient, see e.g.\ Lemma.  
For $i=1$, $\p{1}{G}$ is the usual \emph{$2$-shadow} of $G$ consisting of all pairs $uv$ of vertices such that there exists some edge $e\in E(G)$ with $u,v\in e$.
Let $\CI_{G}(xy)$ be the set of those $i\geq 0$ such that the pair $xy$ is $i$-claimed by $G$, that is,
 \begin{equation}\label{eq:I}
 \textstyle
 \CI_{G}(xy):=\left\{ i\geq 0: \exists\mbox{ distinct $e_1,\dots,e_i\in E(G)$ such that }\big|\{x,y\}\cup (\bigcup_{j=1}^i e_j)\big|\leq ri-2i+2\right\}.
 \end{equation}
More generally, for disjoint subsets $A,B\subseteq \I N$, we say that $G$ \emph{\OutIn{A}{B}-claims} a pair $xy$ if $A\cap \CI_{G}(xy)=\varnothing$ and $B\subseteq \CI_{G}(xy)$. For simplicity, we often omit curly brackets. For example, when $A=\{1\}$ and $B=\{i\}$ we just say \emph{\OutIn{1}{i}-claims}; also, we  let $\pp{1}{i}{G}:=\p{i}{G}\setminus \p{1}{G}$ denote the set of pairs in ${V(G)\choose 2}$ that are \OutIn{1}{i}-claimed by~$G$, and similarly let $\pp{12}{i}{G}:=\p{i}{G}\setminus (\p{1}{G}\cup \p{2}{G})$.

\section{Lower bounds}\label{se:lower}
In order to prove lower bounds in Theorem~\ref{thm:main} and~\ref{thm:lb}, we need the following result proved by Glock, Joos, Kim, K\"{u}hn, Lichev and Pikhurko~\cite{GlockJoosKimKuhnLichevPikhurko24}.

\begin{thm}[\cite{GlockJoosKimKuhnLichevPikhurko24}*{Theorem 3.1}]\label{thm:highgirth}
Fix $k\geq 2$ and $r\geq 3$. Let $F$ be a $\cG{r}{k}$-free $r$-graph. 
%Let $J$ be a supporting graph of $F$ such that the non-edge girth of $(F,J)$ is greater than $k/2$. 
Then,
$$\liminf_{n\rightarrow \infty}\frac{f^{(r)}(n;rk-2k+2,k)}{n^2}\geq \frac{|F|}{2\,|\p{\le \lfloor k/2\rfloor}{F}|},$$
where we define $P_{\le t}(F):=\{xy\in {V(F)\choose 2}\mid  \CI_{F}(xy)\cap [t]\not=\varnothing\}$ to consist of all pairs $xy$ of vertices of $F$ such that $\CI_{F}(xy)$ contains some $i$ with $1\leq i\leq t$.
\end{thm}

In brief, Theorem~\ref{thm:highgirth} is proved by finding, for any large $n$, an almost optimal edge-packing of copies of the graph $J:=\p{\le \lfloor k/2\rfloor}{F}$ in the
$n$-clique, and putting a copy of $F$ ``on top" of each copy of $J$. Since $F$ has no $k$-configuration, it remains to prevent any $k$-configurations that use at least two different copies of $J$. For this, the packing has to be chosen carefully, using the general theory of conflict-free hypergraph matchings developed independently by Delcourt and Postle~\cite{DelcourtPostle24} and Glock, Joos, Kim, K\"{u}hn and Lichev~\cite{24GJKKL}. 

%Before stating their theorem, let us first introduce some necessary definitions.Given an $r$-graph $F$ and a graph $J$, we say that $J$ is a \emph{supporting graph} of $F$ if $V(J)=V(F)$ and $J$ contains the $2$-shadow of~$F$. For such $F$ and $J$, we define the \emph{non-edge girth} of $(F,J)$ to be the smallest $g\ge 1$ such that there exists a $g$-configuration in $F$ whose vertex set contains a non-edge of $J$. In other words, it is largest $g\ge 1$ such that for any $\ell$-configuration $S$ in $F$ with $\ell<g$, every pair in $\binom{V(S)}{2}$ is an edge of $J$. If no such $g$ exists, we set the non-edge girth to be infinity.

We note that Theorem~\ref{thm:highgirth} was used to derive lower bounds of $\pi(r,k)$ for $k\in [4,7]$ in~\cite{GlockJoosKimKuhnLichevPikhurko24,GKLPS24}, as well as to prove the existence of $\pi(r,k)$ in~\cite{DelcourtPostle24,23sh}. In the following, we will also apply this theorem to determine the lower bounds in Theorem~\ref{thm:main} and~\ref{thm:lb}.

\begin{proof}[Proof of the lower bound in Theorem~\ref{thm:main}]
 For $r\ge 4$, the lower bound $\pi(r,8)\ge 1/(r^2-r)$ follows from Theorem~\ref{thm:highgirth} with the $r$-graph $F$ being a single edge $e$ (as then $P_{\le 4}(F)={e\choose 2}$ consists of all pairs inside $e$).    
\end{proof}

\begin{figure}[b]
    \begin{center}
\begin{tikzpicture}
\fill[color=orange,fill opacity=0.2] (0,2.1) -- (1,0.6) -- (-1,0.6) -- cycle;
\fill[color=red,fill opacity=0.2] (1,0.6) -- (-1,-1) -- (1,-1) -- cycle;
\fill[color=blue,fill opacity=0.2] (-1,0.6) -- (-1,-1) -- (1,-1) -- cycle;
    \draw[line width=0.7pt](-1,-1)--(1,-1);
    \draw[line width=0.7pt] (-1,0.6)--(1,0.6);
    \draw[line width=0.7pt] (-1,0.6)--(-1,-1);
    \draw[line width=0.7pt] (1,0.6)--(1,-1);
    \draw[line width=0.7pt] (1,0.6)--(-1,-1);
    \draw[line width=0.7pt] (-1,0.6)--(1,-1);
    \draw[line width=0.7pt] (1,0.6)--(0,2.1);
    \draw[line width=0.7pt] (-1,0.6)--(0,2.1);   
    \node[scale=1.2] at (0,2.37) {$a$};
    \node[scale=1.2] at (1.3,0.6) {$c$};
    \node[scale=1.2] at (-1.3,0.6) {$b$};
    \node[scale=1.2] at (-1.3,-1.2) {$u$};
    \node[scale=1.2] at (1.3,-1.2) {$v$};
    \draw [fill=black] (-1,-1) circle (2pt);
    \draw [fill=black] (1,-1) circle (2pt);
    \draw [fill=black] (-1,0.6) circle (2pt);
    \draw [fill=black] (1,0.6) circle (2pt);
    \draw [fill=black] (0,2.1) circle (2pt);

\end{tikzpicture} 
\caption{An illustration of $R$}
    \end{center}
      \label{fi:1}
\end{figure}

Let us informally describe the construction of $F$ used to prove
Theorem~\ref{thm:lb} via an application of Theorem~\ref{thm:highgirth}. Let $R$ be the $5$-vertex $3$-edge $3$-graph obtained from an edge $abc$ by adding a diamond $\{buv, cuv\}$, see Figure~1.
%~\ref{fi:1}. 
This $3$-graph \OutIn{1}{3}-claims the pairs $au$ and~$av$. Our goal is to construct a $4$-free and $8$-free $3$-graph $F$ consisting of many copies of $R$ such that the number of \OutIn{1}{3}-claimed pairs in $F$ is much smaller than $|F|$. We fix a large integer $m$ and a bipartite graph $G$ with $2m$ vertices, such that $G$ does not contain a 4-cycle as a subgraph and each its vertex has degree $\Theta(\sqrt m)$.  We take a random collection $\C P$ of $2$-paths (that is, paths consisting of 2 edges) where each 2-path of $G$ is included into $\C P$ with probability $p:=(\log m)/\sqrt{m}$, except we remove some (negligibly many) paths to satisfy the property that for any $i\in [8]$,  every $i$-subset $\{P_1,\dots, P_i\}$ of $\C P$ satisfies $v\left(\bigcup_{j\in [i]}P_j\right)\ge i+2$. We construct $F$ from $\C P$ as follows. 
For each path $P_i\in \C P$, say $P_i=\{u_ia_i,a_iv_i\}$, we add two vertices $b_i,c_i$ to $V(F) $ and add all edges from $R_i:=\{a_ib_ic_i,b_iu_iv_i,c_iu_iv_i\}$ to $E(F)$. Thus $R_i$ is a copy of $R$ on top of $P_i$ such that the two edges of $P_i$ are \OutIn{1}{3}-claimed by $R_i$. 

We split the proof of Theorem~\ref{thm:lb} into two main parts. Lemma~\ref{lm:2-paths} returns a collection $\C P$ of $2$-paths with the required properties. Lemma~\ref{lem:const} verifies that the constructed $F$ satisfies all conditions of Theorem~\ref{thm:highgirth}, which will give us the desired lower bound of Theorem~\ref{thm:lb}.

\hide{
A graph $G$ is called \emph{$K$-almost-regular} if $\Delta(G)\le K\delta(G)$, where $\Delta(G)$ is the maximum degree of $G$ and $\delta(G)$ is the minimum degree of $G$. In order to find an almost regular graph, we need the following lemma from~\cite{21CJL}, which is a variant of a result by Erd\H{o}s and Simonovits~\cite{70ES}.

\begin{lemma}[\cite{21CJL}*{Lemma 2.2}]\label{lm:regular}
    Let $\varepsilon,c$ be positive reals, where $\varepsilon<1$. Let $n$ be a positive integer that is sufficiently large as a function of $\varepsilon$ and $c$. Let $G$ be a graph on $n$ vertices with $|G|\ge cn^{1+\varepsilon}$. Then $G$ contains a $K$-almost-regular subgraph $G'$ on $m\ge n^{\frac{\varepsilon-\varepsilon^2}{4+4\varepsilon}}$ vertices such that $|G'|\ge \frac{2c}{5}m^{1+\varepsilon}$ and $K=20\dot 2^{\frac{1}{\varepsilon^2}+1}$.
\end{lemma}

We also need the following lemma from~\cite{16AS}, which can be proved directly by considering a random partition.

\begin{lemma}[\cite{16AS}*{Theorem 2.2.1}]\label{lm:bipartite}
    Let $G$ be a graph with $n$ vertices and $e$ edges. Then $G$ contains a bipartite subgraph with at least $e/2$ edges.
\end{lemma}

Now we are ready to provide the lemma finding a family of 2-paths.}

First, we provide the lemma finding a desired family of 2-paths.

\begin{lemma}\label{lm:2-paths}
For any sufficiently large prime power $q$, there exists a family $\C P$ of $2$-paths satisfying the following properties with $m:=q^2+q+1$:

   \begin{enumerate}[label=(\roman*)]
       \item $|\C P|=\Omega(m^{3/2}\log m)$,
       \item the union graph $\bigcup_{P\in \C P}P$ has 
       %$2m$ vertices and 
       $O(m^{3/2})$ edges,
       \item the union graph $\bigcup_{P\in \C P}P$ is triangle-free, $C_4$-free and $C_5$-free,
       \item for any $i\in [8]$, every $i$-subset $\{P_1,\dots, P_i\}$ of $\C P$ satisfies $v\left(\bigcup_{j\in [i]}P_j\right)\ge i+2$.
   \end{enumerate}
\end{lemma}

\begin{proof}
Let $G'$ be the incidence bipartite graph of the Desarguesian projective plane over the finite field $\mathbbm F_q$. In more detail, let $V_i$, for $i=1,2$, consist of $i$-dimensional linear subspaces of $\mathbbm F_q^3$, and the graph adjancency is the inclusion relation. As it is well-known (and easy to check directly), we have $|V_i|=m$ and each vertex has degree $q+1$ in $G'$.

Let $\C P(G')$ be the set of all $2$-paths of $G'$. We choose every $2$-path in $\C P(G')$ randomly and independently with probability $(\log m)/m^{1/2}$. Denote by $\C P_0\subseteq \C P(G')$ the set of all chosen $2$-paths. Since  
\[
|\C P(G')|=\sum_{v\in V(G')}\binom{d(v)}{2}=\Theta \left(m^2\right),
\]
we have
\begin{equation}\label{eq:p0}
\mathbb{E}\big[\,|\C P_0|\,\big]=\Theta\left(m^2\right)\cdot \frac{\log m}{m^{1/2}}=\Theta\left(m^{3/2}\log m\right),
\end{equation}
which is exactly the order of magnitude in Condition $(i)$. Furthermore, since $\bigcup_{P\in \C P_0}P\subseteq G'$, $\C P_0$ satisfies conditions $(ii)$ and $(iii)$ trivially. To meet condition $(iv)$, we shall use the probabilistic deletion method to remove some $2$-paths from $\C P_0$ which form configurations prohibited in $(iv)$  and show that the number of such removed $2$-paths is $o(m^{3/2}\log m)$.

Given an integer $i$, an $i$-set $\{P_1,\dots,P_i\}\subseteq \C P(G')$ is called \emph{dense} if $v\left(\bigcup_{j\in [i]}P_j\right)\le i+1$, and \emph{sparse} otherwise. We note that every $1$-set $\{P_1\}$ is sparse as $v\left(P_1\right)=3>1+1$. We say that a dense $i$-set $\C S$ is a \emph{minimal dense $i$-set} if every $j$-subset of $\C S$ is sparse for every $j\in [i-1]$. An easy induction argument gives the following proposition.

\begin{prop}\label{pr:minimal}
    Every dense $i$-set with $i\ge 1$ contains a minimal dense $j$-set as a subset for some $1\le j\le i$.
\end{prop}
\hide{\begin{proof}[Proof of Proposition~\ref{pr:minimal}]
    Firstly, every dense $2$-set is minimal since all $1$-sets are sparse. Thus, the statement holds for $i=2$. Now assume that the proposition is true for all $j\le i-1$. Let us consider a fixed arbitrary dense $i$-set $\C S$. If $\C S$ is minimal, then we are done. If not, then $\C S$ contains a dense $k$-set $\C S'$ as a subset for some $k<i$ by the definition. We apply induction step on $\C S'$, and obtain that either $\C S'$ is minimal or $\C S'$ contains a smaller minimal dense subsets. Both cases imply that $\C S$ contains a smaller minimal dense $j$-set as a subset with $j\le k<i$. 
\end{proof}
}

Let us count the number of minimal dense $i$-sets of $\C P(G')$ in the following claim. 

\begin{claim}\label{cl:count}
    Given an integer $i\ge 2$, the number of minimal dense $i$-sets is $O\left(m^{1+\frac{i}{2}}\right)$.
\end{claim}
\begin{proof}[Proof of Claim~\ref{cl:count}]
    Fix an $i\ge 2$. Let $N_i$ denote the number of minimal dense $i$-sets. For $k\le i+1$, we define $N_{i,k}$ to be the number of minimal dense $i$-sets $\{P_1,\dots,P_i\}\subseteq \C P(G')$ satisfying $v\left(\bigcup_{j\in [i]}P_j\right)=k$. Recalling the definition of dense $i$-sets, we have $N_i=\sum_{k\in [i+1]}N_{i,k}$. Hence, it suffices to prove that $N_{i,k}= O\left(m^{1+\frac{i}{2}}\right)$ for every $k\in [i+1]$. 

    Let $k\in [i+1]$ be fixed. Consider a minimal dense $i$-set $\C S:=\{P_1,\dots,P_i\}\subseteq \C P(G')$ with $v\left(\bigcup_{j\in [i]}P_j\right)=k$. We claim that the union graph $\bigcup_{j\in [i]}P_j$ is connected. If not, then $\bigcup_{j\in [i]}P_j$ consists of several connected components, which will give a natural partition of $\C S$ (since every path $P_j\in \C S$, $V(P_j)$ entirely lies in one component). Now assume that $\bigcup_{j\in [i]}P_j$ has $t\ge 2$ components $C_1,\dots, C_t$. For $\ell\in [t]$, let $\C S_\ell\subseteq \C S$ be the proper subset of paths of $\C S$ which entirely lie in the component $C_\ell$. Then there must exist some $\C S_\ell\subseteq \C S$ with $\ell\in [t]$ which is a dense $|\C S_\ell|$-set, since otherwise we can derive the contradiction that
    \[
    v\Big(\bigcup_{j\in [i]}P_j\Big)=\sum_{\ell\in [t]}v\left(C_\ell\right)=\sum_{\ell\in [t]}v\Big(\bigcup_{P\in \C S_\ell}P\Big)\ge \sum_{\ell\in [t]}\left(|\C S_\ell|+2\right)=i+2t>i+2.
    \]
    However, this contradicts to the fact that $\C S$ is a minimal dense $i$-set.

Now we know that each minimal dense $i$-set $\C S:=\{P_1,\dots,P_i\}\subseteq \C P(G')$ with $v\left(\bigcup_{j\in [i]}P_j\right)=k$ corresponds to a connected graph $\bigcup_{j\in [i]}P_j$ on $k$ vertices in $G'$. On the other hand, a $k$-vertex connected subgraph of $G'$ can be the union graph for at most $\binom{k^3}{i}$ minimal dense $i$-sets $\C S:=\{P_1,\dots,P_i\}\subseteq \C P(G')$ with $v\left(\bigcup_{j\in [i]}P_j\right)=k$, since $k^3$ is a trivial upper bound on the number of $2$-paths in this subgraph. Let $\C F_k$ be the set of all $k$-vertex connected subgraphs of~$G'$. Then $N_{i,k}=O\left(|\C F_k|\right)$. 

Let $\C T_k$ be the set of all $k$-vertex trees in $G'$. For every connected $k$-vertex graph $F$ in $\C F_k$, $F$ contains a spanning tree $T\subseteq F$ as a subgraph. On the other hand, every $k$-vertex tree $T\in \C T_k$ can serve as a spanning tree for at most constant number of $k$-vertex connected subgraphs. Therefore, $|\C F_k|= O(|\C T_k|)$. By the regularity of $G'$, we can bound $|\C T_k|$ as follows:
\[
|\C T_k|\le 2m\cdot (k-1)!\cdot  \big(\Delta(G')\big)^{k-1}=O(m\cdot m^{\frac{k-1}{2}})= O(m^{1+\frac{i}{2}}).
\]
This bound is obtained by choosing vertex in $G'$ and then iteratively building a tree by choosing a selected vertex and adding one of its neighbours in $G'$. Thus we conclude that
\[
N_{i,k}= O(|\C F_k|)= O(|\C T_k|)= O(m^{1+\frac{i}{2}}),
\]
as claimed.\end{proof}

Let $X_i$ be the number of minimal dense $i$-sets in which every $2$-paths has been chosen. Then, by Claim~\ref{cl:count}, 
\begin{equation}\label{eq:minimal}
\mathbb{E}\big[\,\sum_{i\in[8]}X_i\,\big]=\sum_{i\in [8]}\mathbb{E}[\,X_i\,]\le \sum_{i\in [8]}\Big(\frac{\log m}{m^{1/2}}\Big)^i\cdot O\left(m^{1+\frac{i}{2}}\right)= O(m\log^8 m).
\end{equation}
Let $\C P\subseteq \C P_0$ be the set obtained from $\C P_0$ by removing one $2$-path from each minimal dense $i$-set with $i\in [8]$. Then $\C P$ inherits properties $(ii)$ and $(iii)$ from $\C P_0$ as $\bigcup_{P\in \C P}P$ is a subgraph of $\bigcup_{P\in \C P_0}P$. Also, it follows from~(\ref{eq:p0}) and~(\ref{eq:minimal}) that 
\[
\mathbb{E}\big[\,|\C P|\,\big]\ge \mathbb{E}\big[\,\C |\C P_0|-\sum_{i\in[8]}X_i\,\big]\ge \Theta(m^{3/2}\log m)-O(m\log ^8m)=\Theta(m^{3/2}\log m).
\]
Take a deterministic outcome $\C P$ with $|\C P|=\Omega(m^{3/2}\log m)$. Then, $\C P$ avoids all minimal dense $i$-sets with $i\in [8]$. By Proposition~\ref{pr:minimal}, $\C P$  avoids all dense $i$-sets for $i\in [8]$, proving Item $(iv)$. This finishes the proof.
\end{proof}

%Let us remark that Lemma~\ref{lm:2-paths} is valid for any large $m$, which can be derived by using Bertrand's postulate and removing a random set of vertices.

The construction of $F$ to be utilized in Theorem~\ref{thm:highgirth} is presented below.

\begin{lemma}\label{lem:const}
    For an infinite sequence of $m$, there exists a $3$-graph $F$ satisfying the following properties:
    \begin{enumerate}[label=(\alph*)]
        \item $F$ is $4$-free and $8$-free,
        \item $F$ is $k^-$-free for every $k\in [2,7]$.
        \item $|F|=\Omega(m^{3/2}\log m)$,
        \item $|\p{\le 4}{F}|=\frac{8}{3}\,|F|+O(m^{3/2})$.
    \end{enumerate}
\end{lemma}

\begin{proof}
    For $m=q^2+q+1$ with sufficiently large prime power $q$, let $\C P=\{P_1,\dots,P_{|\C P|}\}$ be the family of $2$-paths returned by Lemma~\ref{lm:2-paths}. For a $2$-path $P_i\in \C P$, we denote by $a_i$ the internal vertex of $P_i$, and denote by $u_i,v_i$ the two endpoints of $P_i$. Now we construct $F$, using $\C P$. For each $P_i\in \C P$,
    %with $i\in [\,|\C P|\,]$, 
    we add new vertices $b_i$ and $c_i$ to $V(F)$, and add all edges of $R_i:=\{a_ib_ic_i,b_iu_iv_i,c_iu_iv_i\}$ to~$F$. Thus $R_i$ is a copy of $R$ that sits on top of $P_i$; we also say that $P_i$ \emph{supports} $R_i$. Let $F$ be the union of the $3$-graphs $R_i$ for $i\in [\,|\C P|\,]$. We call each such copy of $R$ a \emph{block}.

Let us first calculate $|F|$. We claim that $|F|=3\,|\C P|$. By our construction of $F$, we know that each $R_i$ contributes 3 edges to $E(F)$. It is enough to show that for distinct $i,j\in [\,|\C P|\,]$, we have $E(R_i)\cap E(R_j)=\varnothing$. Indeed, each edge of $R_i$ contains either $b_i$ or $c_i$. Since $b_i,c_i\notin V(R_j)$, no edge of $R_i$ belongs to $E(R_j)$. Thus, Item $(c)$ holds.

To prove Item $(b)$, it suffices to have the following claim.  

\begin{claim}\label{cl:k-free}
    For every $k\in [2,8]$, every $k$ edges of $F$ span at least $k+2$ vertices. 
\end{claim}

\begin{proof}[Proof of Claim~\ref{cl:k-free}]
    Consider an arbitrary set $S:=\{e_1, \dots, e_k\}$ of $k$ edges of~$F$. We can think of $S$ as a subgraph of~$F$. By our construction of $F$, each edge of $S$ belongs to some copy of $R$. Let us assume that edges of $S$ come from $j$ blocks, say $R_1,\dots, R_j$. Then we have $j\le k$ trivially. Each $R_i$, with $i\in [j]$, might contribute $1,2$ or $3$ edges of~$S$. In other words, one can obtain $S$ from $\bigcup_{i\in [j]}R_i$ by trimming $0,1$ or $2$ edges which do not belong to $S$ from each $R_i$. Furthermore, when trimming an edge $e\notin S$, if some vertex $v$ of $e$ does not lie in $V(S)=\cup_{i=1}^k e_i$, then we remove it. 
    
    For every $i\in [j]$, let $t(R_i)$ be the number of trimmed edges of $R_i$, and let $r(R_i)$ be the number of removed vertices of $R_i$. Then $t(R_i)\in \{0,1,2\}$ for every $i\in [j]$. In the proposition below, we give a relation between $t(R_i)$ and $r(R_i)$.

\begin{prop}\label{pr:count-trim}
       For every $R_i$ with $i\in [j]$, it holds that $r(R_i)\le t(R_i)$.
  %     \begin{enumerate}[label=(\arabic*)]
  %         \item If $t(R_i)=1$, then $r(R_i)\le 1$;
   %        \item If $t(R_i)=2$, then $r(R_i)\le 2$;
    %       \item If $t(R_i)=0$, then $r(R_i)=0$.
    %   \end{enumerate}
\end{prop}

\begin{proof}[Proof of Proposition~\ref{pr:count-trim}]
    We split the proof into three cases, depending on whether $t(R_i)=1$, $2$ or $0$. Recall that $R_i:=\{a_ib_ic_i,b_iu_iv_i,c_iu_iv_i\}$. %To prove $(1)$, assume that only one edge is trimmed from $R_i$, which means that the remaining two edges of $R_i$ belong to~$S$. 
    When $t(R_i)=1$ and $a_ib_ic_i$ is removed, both $b_iu_iv_i$ and $c_iu_iv_i$ belong to $S$. Hence, $\{b_i,c_i,u_i,v_i\}\subseteq V(S)$ and $a_i$ is the only possible vertex which may be removed from $V(R_i)$. If $b_iu_iv_i$ (resp.\ $c_iu_iv_i$) is removed and $a_ib_ic_i,c_iu_iv_i\in S$ (resp.\ $a_ib_ic_i,b_iu_iv_i\in S$), then $\{a_i,b_i,c_i,u_i,v_i\}\subseteq V(S)$ and no vertex of $R_i$ can be removed. Thus $r(R_i)\le 1=t(R_i)$.

    When $t(R_i)=2$, the conclusion follows by observing that at least 3 vertices of $R_i$ remain so at most 2 vertices may be removed. 
    %If $u_iv_ib_i, u_iv_ic_i$ are trimmed from $R_i$ and $a_ib_ic_i\in S$, then $\{a_i,b_i,c_i\}\subseteq V(S)$ and hence only $u_i,v_i$ could be removed from $R_i$. Similarly, if $a_ib_ic_i, u_iv_ib_i$ (resp.\ $a_ib_ic_i, u_iv_ic_i$) are trimmed from $R_i$ and $u_iv_ic_i\in S$ (resp.\ $u_iv_ib_i\in S$), then $\{u_i,v_i,c_i\}\subseteq V(S)$ (resp.\ $\{u_i,v_i,b_i\}\subseteq V(S)$) and hence only $a_i,b_i$ (resp.\ $a_i,c_i$) could be removed from $R_i$. Thus $r(R_i)\le 2$, and $(2)$ holds.

    When $t(R_i)=0$, meaning that no edge is trimmed from $R_i$, all three edges  or $R_i$ belong to $S$, and no vertex of $R_i$ can be removed. 
\end{proof}

Since $j\le k\le 8$, we may use Lemma~\ref{lm:2-paths}~(iv) for the $j$-set $\{P_1,\dots, P_j\}$ to obtain that $v(\bigcup_{i\in [j]}P_i)\ge j+2$. Recall that $R_i$ is a copy of $R$ that is based on $P_i$ and its vertex set is obtained by adding two new vertices. Thus, we have
\begin{equation}\label{eq:vtx}
v\big(\bigcup_{i\in [j]}R_i\big)=v\big(\bigcup_{i\in [j]}P_i\big)+2j\ge 3j+2.
\end{equation}

Furthermore, the fact that $E(R_k)\cap E(R_\ell)=\varnothing$ for distinct $k,\ell\in [j]$ yields that $|\bigcup_{i\in [j]}R_i|=3j$. Therefore, in total there are exactly $3j-k$ edges trimmed from $\bigcup_{i\in [j]}R_i$ to obtain~$S$. 
%Set index sets $I_0:=\{i:t(R_i)=0\}$, $I_1:=\{i:t(R_i)=1\}$ and $I_2:=\{i:t(R_i)=2\}$. Then $|I_1|+2|I_2|=3j-k$. 
On the other hand, by Proposition~\ref{pr:count-trim}, the number of vertices removed from $\bigcup_{i\in [j]}R_i$ is at most
\begin{equation}\label{eq:del-vtx}
\sum_{i\in [j]}r(R_i)\le \sum_{i\in [j]}t(R_i)=3j-k.
%=\sum_{i\in I_0}r(R_i)+\sum_{i\in I_1}r(R_i)+\sum_{i\in I_2}r(R_i)\le |I_1|+2|I_2|=3j-k.
\end{equation}
Hence, by~(\ref{eq:vtx}) and~(\ref{eq:del-vtx}), the number of remaining vertices, which is exactly $|V(S)|$, is at least
\begin{equation}\label{eq:k+2}
3j+2-(3j-k)=k+2,
\end{equation}
as desired. This finishes the proof of Claim~\ref{cl:k-free}.
\end{proof}

\noindent\textbf{Remark.} 
In the proof of Claim~\ref{cl:k-free}, for $k\in [2,8]$, if $k$ edges span exactly $k+2$ vertices, it means that all the inequalities in~(\ref{eq:vtx}) and~(\ref{eq:del-vtx}) are equalities. In other words, the corresponding $2$-paths $\{P_1,\dots, P_j\}$ satisfy $v(\bigcup_{i\in [j]}P_i)= j+2$. Also, for those $R_i$ with $t(R_i)=1$ (resp.\ $t(R_i)=2$), we have $r(R_i)=1$ (resp.\ $r(R_i)=2$). In particular, by the proof of Proposition~\ref{pr:count-trim}, if $t(R_i)=1$ and $r(R_i)=1$, then the unique edge we trim from $R_i$ is $a_ib_ic_i$ and the removed vertex of $R_i$ is $a_i$.

Now we turn to Item $(a)$. First, we prove the following helpful proposition.

\begin{prop}\label{pr:trim-1-2}
Suppose that $S$ is a $k$-configuration of $F$ with $3\le k\le 8$ and $S$ consists of edges from $j$ blocks $R_1,\dots ,R_j$. Let $I:=\{i\in [j]:t(R_i)=0\}$ and $R_{I}:=\bigcup_{i\in I}R_i$. Then 
\begin{enumerate}[label=(\Alph*)]
    \item $I\neq\varnothing$,
    \item $R_I$ is a $3|I|$-configuration,
    %of $S$,
    \item $R_I$ can be obtained from $S$ by iteratively deleting edges $e_{k-3|I|}, \dots, e_2, e_{1}$, such that each $e_\ell$ with $\ell \in [\,k-3|I|\,]$ shares exactly two vertices with $R_I\cup \{e_1, \dots, e_{\ell-1}\}$. Furthermore, each $e_\ell$ comes from some block $R_i$ with $t(R_i)\in\{1,2\}$, and $e_\ell$ is $b_iu_iv_i$ or $c_iu_iv_i$. 
\end{enumerate}
\end{prop}

\begin{proof}
By Claim~\ref{cl:k-free}, $S$ spans exactly $k+2$ vertices. Consider any $R_i$ that contributes 1 edge to $S$, that is, satisfies $t(R_i)=2$. If $a_ib_ic_i\in S$ and $b_iu_iv_i, c_iu_iv_i\notin S$ then, since $b_i,c_i$ cannot lie in other edges of $S$ (by our construction of $F$), we can remove the edge $a_ib_ic_i$ from $S$ together with two vertices $b_i,c_i$ and obtain a $k^-$-configuration, which contradicts Claim~\ref{cl:k-free}. Therefore, either $b_iu_iv_i\in S$ or $c_iu_iv_i\in S$. 

For each $R_i$ contributing 2 edges, by the remark after the proof of Claim~\ref{cl:k-free}, we have that $b_iu_iv_i, c_iu_iv_i\in S$ and $a_ib_ic_i\notin S$. For each such $i$, we remove first $b_iu_iv_i$ and then $c_iu_iv_i$ from~$S$. Meanwhile, $b_i$ and the $c_i$ can also be deleted since they do not belong to any other edges of~$S$. Moreover, $u_i,v_i$ must lie in some other edges of the remaining $3$-graph, since otherwise after deleting $b_iu_iv_i$ (or $c_iu_iv_i$) we would gain a $j^-$-configuration with $j\in [k]$, which contradicts Claim~\ref{cl:k-free}. Therefore, during each removal, we lose exactly one edge and one associated vertex. After cleaning all edges from $R_i$ with $t(R_i)=2$, we obtain a $k_0$-configuration $S'$ with exactly $k_0+2$ vertices with $k_0\in [k]$. Then we remove edges from blocks $R_i$ with $t(R_i)=1$ from $S'$. For each $R_i$ with $t(R_i)=1$, we remove $b_iu_iv_i$ and $c_iu_iv_i$ from $S'$. By similar arguments, we conclude that during each edge removal, we lose exactly one edge and one vertex. Indeed, if $|I|\ge 1$, then the remaining $3$-graph is $R_I$ with $3|I|$-edges and $3|I|+2$ vertices. (If $I=\varnothing$, then $R_I$ is the empty 3-graph). We proved Items $(B)$ and $(C)$.

Suppose on the contrary that Item $(A)$ fails, that is $I=\varnothing$, which implies that every block $R_i$ contributes one or two edges. By the argument above, we can enumerate $S=\{e_k, \dots, e_2, e_1\}$ so that each $e_\ell$ with $\ell\in [k]\setminus \{1\}$ shares exactly two vertices with $e_1\cup \dots \cup e_{\ell-1}$. For $e_2$, we have $e_2\cap e_1=2$. If $e_1$ and $e_2$ are from distinct blocks, say $e_1\in R_1$ and $e_2\in R_2$, then $R_1$ intersects $R_2$ on $\{u_1,v_1\}=\{u_2,v_2\}$ as $e_1\in \{b_1u_1v_1,c_1u_1v_1\}$ and $e_2\in \{b_2u_2v_2,c_2u_2v_2\}$. Then the paths $P_1$ and $P_2$, which support $R_1$ and $R_2$ respectively, share two endpoints, giving a $C_4$. The other case is that $e_2,e_1$ come from the same block, say $R_1$, and $e_1=b_1u_1v_1$ and $e_2=c_1u_1v_1$ or vice versa. Then for $e_3$, since $t(R_1)\in \{1,2\}$, $e_3$ is from some other block, say $R_3$, and $e_3\in \{b_3u_3v_3,c_3u_3v_3\}$ shares exactly two vertices with $e_1\cup e_2$. Since all $b_i,c_i$ cannot be shared by distinct blocks, we know that the shared vertices are $\{u_1,v_1\}=\{u_3,v_3\}$. It follows that $P_1$ and $P_3$ form a $C_4$, which is a contradiction to Lemma~\ref{lm:2-paths}~(iii).\end{proof}

Assume that $F$ contains a $4$-configuration $S$ coming from $j$ blocks $R_1,\dots ,R_j$. By Claim~\ref{cl:k-free}, $S$ spans exactly $6$ vertices. By Proposition~\ref{pr:trim-1-2}, there exists some block contributing $3$ edges to~$S$ and thus $j=2$. Without loss of generality, we assume that $t(R_1)=0$. Then, $R_2$ contributes one edge $e\in \{b_2u_2v_2, c_2u_2v_2\}$ to $S$, and $e$ shares two vertices with $R_1$. Since $b_2,c_2$ cannot lie in other blocks, we have $e\cap V(R_1)=\{u_2,v_2\}$. On the other hand, since $b_1,c_1$ cannot lie in other blocks, we also have $e\cap V(R_1)\subseteq \{a_1,u_1,v_1\}$. If $e\ni a_1$ then $P_1\cup P_2$ contains a triangle $a_2u_2v_2$, if $e\cap V(R_1)=\{u_1,v_1\}$ then $P_1\cup P_2$ contains a $4$-cycle $a_1u_1a_2v_1$, see Figure~\ref{fi:4Free}. Both statements contradict Lemma~\ref{lm:2-paths}~(iii). Therefore, $F$ is $4$-free. 
\begin{figure}[t]
    \begin{center}
\begin{tikzpicture}[scale=0.7]
\fill[color=red,fill opacity=0.2 ,draw=black, line width=0.6pt] (0,2.1) -- (1,0.6) -- (-1,0.6) -- cycle;
\fill[color=red,fill opacity=0.2, draw=black, line width=0.6pt] (1,0.6) -- (-1,-1) -- (1,-1) -- cycle;
\fill[color=red,fill opacity=0.2, draw=black, line width=0.6pt] (-1,0.6) -- (-1,-1) -- (1,-1) -- cycle;
 \node[scale=0.8] at (0.5,2.37) {$a_1 (v_2)$};
    \node[scale=0.8] at (1.3,0.55) {$c_1$};
    \node[scale=0.8] at (-1.2,0.55) {$b_1$};
    \node[scale=0.8] at (-1.2,-1.3) {$u_1 (u_2)$};
    \node[scale=0.8] at (1.3,-1.3) {$v_1$};
    \node[scale=0.8] at (-3,-0.7) {$b_2$};
    \node[scale=0.8] at (-2,2.9) {$c_2$};
    \node[scale=0.8] at (-4.3,1.6) {$a_2$};
    \draw[line width=0.7pt,dashed] (0,2.1)--(1,-1);
    \draw[color=blue,line width=0.7pt,dashed] (-4,1.6)--(-1,-1);
    \draw[color=blue,line width=0.7pt,dashed] (-4,1.6)--(0,2.1);
\fill[color=red,fill opacity=0.2, draw=black, line width=0.6pt] (-1,-1) -- (-3,-0.4) -- (0,2.1) -- cycle;
\fill[color=red,fill opacity=0, draw=black, line width=0.6pt] (-1,-1) -- (0,2.1) -- (-2,2.7) -- cycle;
\fill[color=red,fill opacity=0, draw=black, line width=0.6pt] (-4,1.6) -- (-3,-0.4) -- (-2,2.7) -- cycle;
 \draw[color=blue,line width=0.7pt] (0,2.1)--(-1,-1);

\draw[color=blue,line width=0.7pt,dashed] (6,3.6)--(5,0.5);
\draw[color=blue,line width=0.7pt,dashed] (6,3.6)--(7,0.5);
\draw[color=blue,line width=0.7pt,dashed] (6,-2.4)--(7,0.5);
\draw[color=blue,line width=0.7pt,dashed] (6,-2.4)--(5,0.5);
\fill[color=red,fill opacity=0.2 ,draw=black, line width=0.6pt] (6,3.6) -- (7,2.1) -- (5,2.1) -- cycle;
\fill[color=red,fill opacity=0.2, draw=black, line width=0.6pt] (7,2.1) -- (5,0.5) -- (7,0.5) -- cycle;
\fill[color=red,fill opacity=0.2, draw=black, line width=0.6pt] (5,2.1) -- (5,0.5) -- (7,0.5) -- cycle;  
\fill[color=red,fill opacity=0.2, draw=black, line width=0.6pt] (5,-0.9) -- (5,0.5) -- (7,0.5) -- cycle;
\fill[color=red,fill opacity=0, draw=black, line width=0.6pt] (7,-0.9) -- (5,0.5) -- (7,0.5) -- cycle;  
\fill[color=red,fill opacity=0, draw=black, line width=0.6pt] (5,-0.9) -- (7,-0.9) -- (6,-2.4) -- cycle;
\node[scale=0.8] at (6,3.8) {$a_1$};
\node[scale=0.8] at (4.7,2.1) {$b_1$};
\node[scale=0.8] at (7.3,2.1) {$c_1$};
\node[scale=0.8] at (4.3,0.5) {$(u_2) u_1$};
\node[scale=0.8] at (7.7,0.5) {$v_1 (v_2)$};
\node[scale=0.8] at (4.7,-0.9) {$b_2$};
\node[scale=0.8] at (7.3,-0.9) {$c_2$};
\node[scale=0.8] at (6,-2.6) {$a_2$};

\end{tikzpicture} 
\caption{An illustration of the proof that $F$ is $4$-free}\label{fi:4Free}
    \end{center}
\end{figure}

Suppose that $F$ contains an $8$-configuration $Q$ consisting of edges from $j$ blocks $R_1,\dots ,R_j$. By Proposition~\ref{pr:trim-1-2}, we have $1\le |I|\le 2$. Let $e_1,\dots, e_5$ be the sequence of edges returned by Proposition~\ref{pr:trim-1-2} $(C)$. First assume that $|I|=1$ and $t(R_1)=0$. Without loss of generality, suppose that $e_1$ belongs to $R_2$; thus $e_1\in \{b_2u_2v_2,c_2u_2v_2\}$. Since $b_2,c_2$ cannot lie in other blocks, we derive that $e_1\cap R_1=\{u_2,v_2\}$. Similarly to the arguments in the proof of $4$-freeness, we have $e\cap V(R_1)\subseteq \{a_1,u_1,v_1\}$. If $e\cap V(R_1)=\{a_1,u_1\}$ or $e\cap V(R_1)=\{a_1,v_1\}$, then $P_1\cup P_2$ contains a triangle $a_2u_2v_2$. If $e\cap V(R_1)=\{u_1,v_1\}$, then $P_1\cup P_2$ contains a $4$-cycle $a_1u_1a_2v_1$. Both cases contradict Lemma~\ref{lm:2-paths}~(iii). 

The remaining case is that $|I|=2$. Assume that $t(R_1)=t(R_2)=0$. Then, by Proposition~\ref{pr:trim-1-2} and Claim~\ref{cl:k-free}, $R_1\cup R_2$ is a $6$-configuration on $8$ vertices. Thus, $R_1$ shares two vertices with $R_2$. Again $b_1,c_1,b_2,c_2$ do not belong to other blocks. Therefore, $V(R_1)\cap V(R_2)\subseteq \{a_1,u_1,u_1\}$ and $V(R_1)\cap V(R_2)\subseteq \{a_2,u_2,u_2\}$. Note that $V(R_1)\cap V(R_2)$ cannot be $\{u_1,v_1\}$ or $\{u_2,v_2\}$, as otherwise $P_1\cup P_2$ contains a triangle or $C_4$. This means that $P_1$ and $P_2$ share one edge, and thus $P_1\cup P_2$ is a $3$-path or a star. In the notation of Proposition~\ref{pr:trim-1-2} $(C)$, 
assume that $e_1$ comes from $R_3$. By Proposition~\ref{pr:trim-1-2} $(C)$, we have $e_1\in \{b_3u_3v_3,c_3u_3v_3\}$. Then, $e_1$ shares two vertices $u_3,v_3$ with $R_1\cup R_2$ and $u_3,v_3\in \{a_1,a_2,u_1,v_1,u_2,v_2\}$. This implies that the endpoints of $P_3$ lie in $V(P_1\cup P_2)$. One can easily check $P_1\cup P_2\cup P_3$ must contain a triangle, $C_4$ or $C_5$, which is a contradiction to Lemma~\ref{lm:2-paths}~(iii). Thus, $F$ is $8$-free. We proved Item $(a)$ of Lemma~\ref{lem:const}.

Let us turn to Item $(d)$. Observe that every $3$-configuration in $F$ comes from one block by Proposition~\ref{pr:trim-1-2} $(A)$. Let us show that the same also holds for any
$2$-configuration. Suppose on the contrary that a $2$-configuration $\{e_1,e_2\}$ satisfies, say, $e_1\in R_1$ and $e_2\in R_2$ for distinct $R_1, R_2$. Then, $|e_1\cap e_2|=2$, which implies that $e_1\in \{b_1u_1v_1, c_1u_1v_1\}$ and $e_1\in \{b_1u_1v_1, c_1u_1v_1\}$. Again, since $b_1,b_2,c_1,c_2$ cannot lie in other blocks, we have $e_1\cap e_2=\{u_1,v_1\}=\{u_2,v_2\}$. This gives a $C_4$ in $P_1\cup P_2$, a contradiction. 

Recall that $F$ consists of $|\C P|=|F|/3$ copies of $R$. Thus, $\p{1}{F}=\bigcup_{i\in [\,|\C P|\,]}\p{1}{R_i}$. Each block $R$ satisfies $|\p{1}{R}|=8$. Furthermore, $\p{1}{R_i}\cap \p{1}{R_j}=\varnothing$ for every distinct $i,j\in [\,|\C P|\,]$, since otherwise we would obtain a $2$-configuration from distinct blocks. Hence, we have $|\p{1}{F}|=8|\C P|=\frac{8}{3}|F|$. Moreover, we claim that $\pp{1}{2}{F}=\varnothing$. Indeed, every $2$-configuration lies in a single block and by our construction, every $2$-claimed pair is also $1$-claimed. Now let us consider $\pp{1}{3}{F}$. By the arguments above, we know that every $3$-configuration is a block. For a block $R_i$, only $a_iu_i$ and $a_iv_i$ are \OutIn{1}{3}-claimed, and these two pairs are edges of the union graph $\bigcup_{P\in \C P}P$. Consequently, $|\pp{1}{3}{F}|\le |\bigcup _{P\in \C P}P|\le O(m^{3/2})$ by Lemma~\ref{lm:2-paths}. At the end, since $F$ is $4$-free, there is no pair $4$-claimed by~$F$. We conclude that $\p{\le 4}{F}=\p{\le 3}{F}=\frac{8}{3}|F|+O(m^{3/2})$. 

This finishes the proof of Lemma~\ref{lem:const}.    
\end{proof}

\begin{proof}[Proof of Theorem~\ref{thm:lb}]
   Let $F$ be the $3$-graph given by Lemma~\ref{lem:const} for some $m$. Then $F$ is $\cG{3}{8}$-free. By Theorem~\ref{thm:highgirth}, we obtain
   \[
   \liminf_{n\rightarrow \infty}\frac{f^{(3)}(n;10,8)}{n^2}\geq \frac{|F|}{2\,|\p{\le 4}{F}|}\ge \frac{3}{16}\cdot \frac{|F|}{|F|+O(m^{3/2})}.
   \]
   Since $|F|=\Omega(m^{3/2}\log m)=\omega(m^{3/2})$, we have $\pi (3,8)\ge \frac{3}{16}$ by taking $m\to \infty$.

\end{proof}

\section{Proof of the upper bound in Theorem~\ref{thm:main}}\label{se:upper}

For the upper bound, we will use the following lemma, proved by Delcourt and Postle~\cite{DelcourtPostle24}*{Theorem 1.7} for $r=3$ and by Shangguan~\cite{23sh}*{Lemma 5} for $r\ge 4$, which enables us to exclude smaller ``denser'' structures when considering the limit.

\begin{lemma}[\cite{23sh}*{Lemma 5}]\label{lem:ub}
    For all fixed $r\ge 3$ and $k\ge 3$,
    $$\limsup_{n\to \infty}\frac{f^{(r)}(n;rk-2k+2,k)}{n^2}\le \limsup_{n\to \infty}\frac{\ex (n,\cG{r}{k})}{n^2}.$$
\end{lemma}

Lemma~\ref{lem:ub} gives that, when proving an upper bound on $\pi(r,8)$, it suffices to consider only $\cG{r}{8}$-free $r$-graphs.

%Inspired by the work of Glock, Joos, Kim, K\"{u}hn, Lichev and Pikhurko~\cite{GlockJoosKimKuhnLichevPikhurko24} for $k=4$ and Glock, Kim, Lichev, Pikhurko and Sun~\cite{GKLPS24} for $k\in \{5,6,7\}$, our proof of the upper bound relies on a merging operation of edges along with a weight distribution of pairs. 
Let us now briefly outline our proof strategy, which was inspired by~\cite{GlockJoosKimKuhnLichevPikhurko24,GKLPS24}. Assume that $G$ is an $n$-vertex $\cG{r}{8}$-free $r$-graph. Starting with the trivial edge partition of $E(G)$ into single edges, we apply some merging rules and iteratively merge parts into larger clusters. The $\cG{r}{8}$-freeness allows us to describe the possible structure of each final cluster. We then assign weights to some vertex pairs within each final cluster (which is a subgraph of $G$). During this step, we ensure that every pair in $\binom{V(G)}{2}$ receives a total weight at most 1. This method enables us to derive an upper bound  on $|G|$ by upper bounding the ratio of the number of edges to the sum of weights for each possible cluster.

We split the proof of the upper bound in Theorem~\ref{thm:main} into two parts. First, we will introduce the merging operation and analyse the structure of the final edge partition. In the second part, we assign weights and finalise the proof.

\subsection{Merging and analysing}

\subsubsection{Some common definitions and results}\label{se:UpperCommon}
Recall that for an $r$-graph $F$ and a pair $uv$, the set $\CI_F(uv)$ 
%(defined in~\eqref{eq:I}) 
consists of all integers $j\ge 0$ such that $F$ has $j$ edges that together with $uv$ include at most $rj-2j+2$ vertices. 
Note that, by definition, 0 always belongs to $\CI_{G}(uv)$, which is notationally convenient in the statement of the following easy but very useful observation.

\begin{lemma}[\cite{GKLPS24}*{Lemma 5.1}]\label{lm: 2} For any $\C F^{(r)}(rk-2k+2,k)$-free $r$-graph $G$, any $uv\in {V(G)\choose 2}$ and any edge-disjoint subgraphs $F_1,\dots,F_s\subseteq G$, the  sum-set $\sum_{i=1}^s \CI_{F_i}(uv)=\left\{\sum_{i=1}^s m_i\mid m_i\in \CI_{F_i}(uv)\right\}$ does not contain~$k$.
\end{lemma}

\hide{
For $r,k\geq 3$, let $g^{(r)}(n;rk-2k+2,k)$ denote the maximum number of edges in an $n$-vertex $r$-graph $G$ such that $G$ is $(rk-2k+2,k)$-free and also $(r\ell-2\ell +1,\ell)$-free for all $\ell\in \{2,\ldots,k-1\}$. 

\begin{lemma}{\rm{(}\cite[Lemma 5]{sh}\rm{)}}.\label{chong}
    For all fixed $r\geq 3$ and $k\geq 3$,
    $$\limsup_{n\rightarrow \infty}\frac{f^{(r)}(n;rk-2k+2,k)}{n^2}\leq \limsup_{n\rightarrow \infty}\frac{g^{(r)}(n;rk-2k+2,k)}{n^2}.$$
\end{lemma}
}

As we mentioned in the introduction, our proof strategy to bound the size of an $(rk-2k+2,k)$-free $r$-graph $G$ from above is to analyse possible isomorphism types of the parts of some partition of $E(G)$ which is obtained from the trivial partition into single edges by iteratively applying some merging rules. 
We build the final partition in stages (with each stage having a different merging rule) as the intermediate families are also needed in our analysis. 
Let us now develop some general notation and prove some basic results related to merging. 

Let $G$ be an arbitrary $r$-graph. When dealing with a partition $\C P$ of $E(G)$, we will view each element $F\in\C P$ as an $r$-graph whose vertex set is the union of the edges in~$F$.
Let $A,B\subseteq \I N$ be any (not necessarily disjoint) sets of positive integers. For two subgraphs $F,H\subseteq G$, if they are edge disjoint and there is a pair $uv$ such that $A\subseteq \CI_F(uv)$ and $B\subseteq \CI_H(uv)$, then we say that $F$ and $H$ are \emph{$\OMerge{A}{B}$-mergeable (via $uv$)}. Note that this relation is not symmetric in $F$ and $H$: the first (resp.\ second) $r$-graph $A$-claims (resp.\ $B$-claims) the pair~$uv$.
When the ordering of the two $r$-graphs does not matter, we use the shorthand \emph{$\Merge{A}{B}$-mergeable} to mean $\OMerge{A}{B}$-mergeable or $\OMerge{B}{A}$-mergeable.
For a partition $\C P$ of $E(G)$, its \emph{$\Merge{A}{B}$-merging} is the partition $\C M_{\Merge{A}{B}}(\C P)$ of $E(G)$ obtained from $\C P$ by iteratively and as long as possible taking a pair of distinct $\Merge{A}{B}$-mergeable parts in the current partition
and replacing them by their union.
Note that the final partition $\C M_{\Merge{A}{B}}(\C P)$ is a coarsening of $\C P$ and contains no $\Merge{A}{B}$-mergeable pairs of $r$-graphs. 
When $\C P$ is clear from the context, we may refer to the elements of $\C M_{\Merge{A}{B}}(\C P)$ as \emph{$\Merge{A}{B}$-\cluster{}s}. 
Likewise, a subgraph $F$ of $G$ that can appear as a part in some intermediate stage of the $\Merge{A}{B}$-merging process starting with $\C P$ is called a \emph{partial $\Merge{A}{B}$-\cluster{}} and we let $\C M'_{\Merge{A}{B}}(\C P)$ denote the set of all partial $\Merge{A}{B}$-\cluster{}s. In other words, $\C M'_{\Merge{A}{B}}(\C P)$ is the smallest family of $r$-graphs which contains $\C P$ as a subfamily and is closed under taking the union of $\Merge{A}{B}$-mergeable elements. The monotonicity of the merging rule implies that $\C M_{\Merge{A}{B}}(\C P)$ is exactly the set of maximal (by inclusion) elements of $\C M'_{\Merge{A}{B}}(\C P)$ and that the final partition $\C M_{\Merge{A}{B}}(\C P)$ does not depend on the order in which we merge parts. In other words, since the merging rule is monotone, each merge operation only enlarges parts without affecting the mergeability of others, so the process defines a closure under union, whose set of maximal elements is unique regardless of merge order.

In the frequently occurring case when $A=\{1\}$ and $B=\{j\}$, we abbreviate $\OMerge{\{1\}}{\{j\}}$ to $\OMerge{}{j}$ and $\Merge{\{1\}}{\{j\}}$ to $\Merge{}{j}$ in the above nomenclature.
Thus, $\OMerge{}{j}$-mergeable (resp.\ $\Merge{}{j}$-mergeable) means $\OMerge{\{1\}}{\{j\}}$-mergeable (resp.\ $\Merge{\{1\}}{\{j\}}$-mergeable). 

As an example, let us look at the following merging rule that is actually used as the first step in  our proof of the upper bound. Namely, given $G$, let 
$$
 \OneClusters:=\C M_{\Merge{\{1\}}{\{1\}}}(\C P_{\mathrm{trivial}})
$$ 
 be the $\Merge{}{1}$-merging  of the trivial partition $\C P_{\mathrm{trivial}}$ of $G$ into single edges. We call the elements of $\OneClusters$ \emph{\OneCluster{}s}. Here is an alternative description of~$\OneClusters$. 
Call a subgraph $F\subseteq G$ \emph{connected} if for any two edges $X,Y\in F$ there is a sequence of edges $X_1=X,X_2,\ldots,X_m=Y$ in $F$ such that, for every $i\in [m-1]$, we have $|X_i\cap X_{i+1}|\geq 2$. 
%In other words, inside $F$ any edge can be reached from any other by a chain of edges with every two consecutive edges overlapping in at least two vertices.  
 Then, \OneCluster{}s are exactly maximal connected subgraphs of $G$  (and partial \OneCluster{}s are exactly connected subgraphs).

We will also use (often without explicit mention) the following result, which is a generalisation of the well-known fact that we can remove edges from any connected 2-graph one by one, down to any given connected subgraph, while keeping the edge set connected. 
The assumption of this result, roughly speaking, is that the merging process cannot create any new mergeable pairs.

\begin{lemma}[Trimming Lemma~\cite{GKLPS24}*{Lemma~5.2}]\label{lm:Trim} 
Fix an $r$-graph $G$, a partition $\C P$ of $E(G)$ and sets $A,B\subseteq \I N$.
Suppose that, for all $\OMerge{A}{B}$-mergeable (and thus edge-disjoint) $F,H\in\C M'_{\Merge{A}{B}}(\C P)$,
there exist $\OMerge{A}{B}$-mergeable  $F',H'\in\C P$ such that $F'\subseteq F$ and $H'\subseteq H$.

Then, for every partial $\Merge{A}{B}$-\cluster{}s
$F_0\subseteq F$, there is an ordering $F_1,\dots,F_s$ of the elements of $\C P$ that lie inside $F\setminus F_0$ such that, for every $i\in [s]$, $\bigcup_{j=0}^{i-1} F_j$ and $F_i$ are $\Merge{A}{B}$-mergeable (and, in particular, $\bigcup_{j=0}^{i} F_j$ is a partial $\Merge{A}{B}$-\cluster{} for every $i\in [s]$).
\end{lemma}

In the special case $A=B=\{1\}$ (when partial \cluster{}s are just connected subgraphs), the assumption of Lemma~\ref{lm:Trim} is vacuously true. Since we are going to use its conclusion quite often,
we state it separately.

\begin{cor}\label{cr:Trim} For every pair $F_0\subseteq F$ of connected $r$-graphs, there is an ordering $X_1,\dots,X_s$ of the edges in $F\setminus F_0$ such that, for every $i\in [s]$, the $r$-graph $F_0\cup \{X_1,\dots,X_i\}$ is connected.\end{cor}

We say that an $r$-graph is a \emph{$1$-tree} if it contains only one edge. 
For $i\geq 2$, we recursively define an \emph{$i$-tree} as any $r$-graph that can be obtained from an $(i-1)$-tree $T$ by adding a new edge that consists of a pair $ab$ in the $2$-shadow $\p{1} T$ of $T$ and $r-2$ new vertices (not present in~$T$). Clearly, every $i$-tree is connected. Like the usual 2-graph trees, $i$-trees are the ``sparsest" connected $r$-graphs of given size. Any $i$-tree $T$ satisfies
\begin{equation}\label{eq:tree}
 |\p{1} T|=i{r\choose 2}-i+1\quad\mbox{and}\quad |\pp{1}{2} T|\geq (i-1) (r-2)^2.
 %,\quad \mbox{and,}\quad \mbox{if $i\geq 3$ then } t(T)= \pp{1}{2} T). 
\end{equation}
(Recall that $\pp{1}{2} T$ is the set of pairs which are $2$-claimed but not \OneClaim{}ed by~$T$.) Note that the second inequality in~\eqref{eq:tree} is equality if, for example, $T$ is an \emph{$i$-path}, meaning that we can order the edges of  the $i$-tree $T$ as $X_1,\dots,X_i$ so that, for each $j\in [i-1]$, the intersection $X_{j+1}\cap (\bigcup_{s=1}^j X_s)$ consists of exactly one pair of vertices and this pair belongs to $\p{1}{X_j}\setminus \p{1}{\bigcup_{s=1}^{j-1} X_s}$.

The following result shows that the \OneCluster{}s of any $\cG{r}{k}$-free graph have a very simple structure: namely, they are all small trees.

\begin{lemma}[\cite{GKLPS24}*{Lemma~5.4}]\label{lm:Trees} 
With the above notation, if $G$ is $\cG{r}{k}$-free, then every $F\in\OneClusters$ is an $m$-tree for some $m\in [k-1]$.
\end{lemma} 

\subsubsection{Structural lemma for $\C M_2$}

Given a $\cG{r}{8}$-free $r$-graph $G$, we consider another edge partition
\[
\C M_2:= \C M_{\Merge{\{1\}}{\{2\}}}(\C M_1),
\]
which is obtained from $\C M_1$ by iteratively merging $(2)$-mergeable pairs. Let $\C M'_2:=\C M'_{\Merge{\{1\}}{\{2\}}}(\C M_1)$ be the set of all partial $\Merge{\{1\}}{\{2\}}$-clusters. For simplicity, we call the elements of $\C M_2$ (resp.~$\C M'_2$) {\emph{2-clusters}} (resp.~{\emph{partial 2-clusters}}).

In Lemma~\ref{lem:str} below, we provide some combinatorial properties of 2-clusters, which will allow us to assign appropriate weights and then achieve correct upper bound in the next section. Before moving to the structural lemma, we first introduce some definitions. 

For $F\in \C M'_2$ which is made by merging 1-clusters $F_1,\dots ,F_m$ in this order as in Lemma~\ref{lm:Trim}, we call the sequence $(|F_1|,\dots ,|F_m|)$ of sizes a \emph{composition} of~$F$.  Its non-increasing reordering is called \emph{the (non-increasing) composition} of~$F$. 
%For two $r$-graphs $F_1$ and $F_2$, we call that $F_1$ {\emph {attaches}} to $F_2$ 
%(or $F_2$ {\emph {sticks}} to $F_1$) 
%if there is a diamond of $F_1$ claiming a pair covered by $F_2$.
Given an $r$-graph $F$, we say that an edge $e\in E(F)$ is {\emph {flexible}} in $F$ if there are $r-2$ vertices $\{v_1,\dots ,v_{r-2}\}\subseteq e$ such that, for each $i\in [r-2]$, $v_i$ does not belong to any other edge of~$F$. We also call the set of these $r-2$ vertices $\{v_1,\dots ,v_{r-2}\}$ a {\emph{flexible set}}. This means that if we delete a flexible set from $V(F)$, then $F$ will lose only one edge. In the following, we will refer to ``removing/deleting'' a flexible edge $e$ from $F$ as removing edge $e$ and all $r-2$ vertices in its flexible set. 
%In other words, we obtain a subgraph of $F$ with edge set $E(F)\setminus \{e\}$ and the vertex set is the union of the remaining edges. 
We denote by $Q(F)\subseteq E(F)$ the set of flexible edges of $F$, and let $q(F):=|Q(F)|$.

\begin{lemma}\label{lm:flex-edge}
    For every $m$-tree $F$ with $m\ge 2$, we have $q(F)\ge 2$.
\end{lemma}

\begin{proof}
    We use induction on $m$. For $m=2$, a $2$-tree is a pair of edges sharing exactly two vertices. Thus, both edges are flexible. Now consider an $m$-tree~$F$ with $m\ge 3$. Recall that $F$ can be obtained from an $(m-1)$-tree, say $F'$, by adding a new edge $e$ which consists of a pair $ab$ in the $P_1(F')$ and $r-2$ new vertices. Therefore, $e\in Q(F)$. On the other hand, by induction hypothesis, $F'$ contains two flexible edges $e_1,e_2$. Let $S_1\subseteq e_1,$ and $S_2\subseteq e_2$ be the flexible sets of $e_1$ and $e_2$ respectively. Then $S_1\cap e_2=S_2\cap e_1=\varnothing$, by the definition of flexible set. We claim that the new added edge $e$ cannot intersect both $S_1$ and $S_2$. If it intersects both, then one of the two intersection vertices $a,b$ lies in $S_1$ and the other in $S_2$. Since $ab\in P_1(F')$, the pair $ab$ is covered by some edge $e_3\in E(F')\setminus \{e_1,e_2\}$, which contradicts that $S_1,S_2$ are flexible sets of $F'$. Hence, without loss of generality, we can assume that $e\cap S_1=\varnothing$. As a consequence, $e_1$ is also a flexible edge of $F$ and $q(F)\ge 2$.    
\end{proof}

One big challenge in proving tight upper bounds for $k=8$ (also present  for smaller $k$) is that there may be \emph{exceptional} 2-clusters, that is, $2$-clusters that contain more than $k$ edges (while being $\C G_k^{(r)}$-free). We will now describe some families of 2-clusters that, as claimed by the first part of Lemma~\ref{lem:str},  include every exceptional one. Recall that if we do $\OMerge{A}{B}$-merging then $A$ refers to the $r$-graph mentioned first (and $B$ to the $r$-graph mentioned later).

Let $\mathcal{A}$ be the family of $9$-edge 2-clusters with a composition $(5,2,2)$ which can be obtained as follows. Take any $5$-tree $S_1$ in $\C M_1$ with exactly two flexible $(r-2)$-sets $A_1$ and $A_2$. (In fact, every such tree  $S_1$ can be shown to be a path but we will not need this.)
Then $\OMerge{1}{2}$-merge $S_1$ with diamonds $D_1, D_2\in \C M_1$ so that the pair $x_iy_i$ in $\p{1}{S_1}$ $2$-claimed by the diamond $D_i$ intersects $A_i$ for $i=1,2$. 
%Thus any 2-cluster in this family has no flexible edges.

%We say a $2i$-edge graph $F$ is a {\emph{i-diamond-path}} if it is the union of $i$ diamonds $D_1,\dots ,D_i$ such that $|V(D_j)\cap V(\cup _{k\in [j-1]}D_k)|=2$ for $2\le j\le i$. 

Let $\mathcal{B}$ be the family of $9$-edge 2-clusters $B$ with a composition $(1,2,4,2)$ that can be obtained as follows. We start with a single $1$-tree $S_1$, that is, with a single edge. Then we $\OMerge{1}{2}$-merge $S_1$ with a $2$-tree (diamond) $S_2$ and with a $4$-tree $S_3$ via any two distinct pairs in $\p{1}{S_1}$. Then we $\OMerge{1}{2}$-merge $S_3$ with a diamond $S_4$. 

Let us remark that, since a $9$-edge 2-cluster in a $\mathcal{G}_8^{(r)}$-free $r$-graph cannot have a flexible edge, we can describe the structure of $B\in\mathcal{B}$ more precisely. For example, it must hold that $q(S_3)=2$ and each of the two flexible $(r-2)$-sets of $S_3$ is made inflexible as the result of the two mergings involving $S_3$. However, here (and later) we prefer to give simple descriptions that, even if less precise, are nonetheless enough for our proof.

%Additionally, $T_3$ consists of the diamond $D$ and a $2$-path $P_2$ sharing one unique shadow with $D$ (which means that there is only one flexible edge $e$ in $T_1\cup D_2\cup T_3$), and $D_4$ attaches to the edge $e$.

Let $\mathcal{C}_1$ be the family of $9$-edge 2-clusters with a composition $(3,2,2,2)$, which can be obtained from a $3$-tree $S$ by 3 times iteratively $\Merge{1}{2}$-merging the current $r$-graph and a new diamond.

Let $\mathcal{C}_2$ be the family of $11$-edge 2-clusters with a composition $(3,2,2,2,2)$ which can be obtained by taking a one element $F$ of $\mathcal{C}_1$ and $\OMerge{1}{2}$-merging it with a new diamond. 

Let $\mathcal{E}$ be the set of $9$-edge 2-clusters with a composition $(3,1,1,2,2)$ that can be obtained as follows. We start with a $1$-tree $S_1$ and $\OMerge{1}{2}$-merge $S_1$ via distinct pairs of $\p{1}{S_1}$ with a diamond $S_2$ and a $3$-tree $S_3$. Then we $\OMerge{2}{1}$-merge $S_1\cup S_2\cup S_3$  with another $1$-tree $S_4$. A last diamond $S_5$ is $\OMerge{2}{1}$-merged with $S_4$.

%By Lemma~\ref{lm:Trim}, let us re-order the sequence starting from a $1$-tree $S_1=e_1$. By Claim~\ref{cl:overlap}, there exists a tree $S_2$ $\OMerge{2}{1}$-merging $S_1$. Since $e_1$ is still a flexible edge of $S_1\cup S_2$, we use Claim~\ref{cl:overlap} again and obtain a $S_3$ $\OMerge{2}{1}$-merging $S_1\cup S_2$ via a pair in $e_1$. The following \OneCluster{} $S_4$ in the sequence must be the $1$-tree $\OMerge{1}{2}$-merging $S_1\cup S_2\cup S_3$. Since $3$-tree is not the last one, one of $S_2$ and $S_3$ is the $3$-tree. Assuming $S_2$ is a diamond and $S_3$ is a $3$-tree, then there is an edge $e$ of $S_3$ such that $e$ is a flexible edge of $S_1\cup S_2\cup S_3$. Since $S_1=e_1$ cannot be the last cluster, we derive that $S_4$ $\OMerge{1}{2}$-merges $S_1\cup S_2\cup S_3$ via a pair intersecting $e$, and the subsequent $S_5$ would $\OMerge{2}{1}$-merge $S_4$. By $\cG{r}{8}$-freeness, no pair in $\p{2}{S_5}\setminus \{u_5v_5\}$ is $1$-claimed or $2$-claimed by $S_1\cup S_2\cup S_3\cup S_4$ and $F'$ cannot $\Merge{1}{2}$-merge more trees. Therefore $F'=F\in \mathcal{E}$ and (\ref{eq:claim}) holds.

Let $\mathcal{F}$ be the set of $9$-edge 2-clusters with a composition $(1,2,2,2,2)$ which can be obtained as follows. We can start with a single $1$-tree $S_1$ and $\OMerge{1}{2}$-merge $S_1$ via distinct pairs of $\p{1}{S_1}$  with two diamonds $D_2, D_3$. Then $\OMerge{2}{1}$-merge $S_1\cup D_2\cup D_3$   with a new diamond $D_4$ so that the common pair is in $\pp{1}{2}{D_2}\cup \pp{1}{2}{D_3}$ and belongs to exactly one edge of $D_4$.
The last diamond $D_5$ is $\OMerge{2}{1}$-merged with $D_4$ via a pair in $\p{1}{e}$, where $e$ is the unique flexible edge of $S_1\cup D_2\cup D_3\cup D_4$ (which belongs to $D_4$).

Let $\mathcal{S}_i$ be the family of $(2i+1)$-edge 2-clusters obtained from a single $1$-tree $S_1$ by $i$ times iteratively $\OMerge{1}{2}$-merging the current $r$-graph with diamonds $D_1,\dots ,D_i$. Thus each element of $\mathcal{S}_i$ has a composition with one entry being $1$ and $i$ entries being $2$.

See Figure~\ref{fi:families} for  some illustrations of the above families.

\begin{figure}
\centering
\renewcommand{\arraystretch}{1.3}
    \begin{scriptsize}
  \begin{tabular}[b]{|>{\centering\arraybackslash}m{5cm}|>{\centering\arraybackslash}m{5cm}|}%{\centering\arraybackslash}m{4cm}|}    
\hline
    {
    \begin{tikzpicture}[scale=0.7,line cap=round,line join=round,x=1cm,y=1cm,baseline={(0,-1)}]
\foreach \i [count=\n from 1] in {0,...,4} {
    \filldraw[blue, fill opacity={0.3}, draw=black, line width=0.4pt]
        (\i,0.5) rectangle ++(1,-1);
}
%\foreach \x in {0,...,5}
 %   \foreach \y in {0.5,-0.5}
  %      \filldraw[black] (\x,\y) circle (1.2pt);
\filldraw[fill=orange, fill opacity=0.3, draw=black, line width=0.5pt]
  (4,0.5) -- (4,1.5) -- (4.5,2) -- (4.5,1) -- cycle;
\filldraw[fill=orange, fill opacity=0.3, draw=black, line width=0.5pt]
(5,0.5) -- (5,1.5) -- (4.5,2) -- (4.5,1) -- cycle;
%\foreach \p in {(4,0.5),(4,1.5),(4.5,2),(4.5,1),(5,0.5),(5,1.5)}{
%  \filldraw[black] \p circle (1.2pt);
%}
\filldraw[fill=cyan, fill opacity=0.3, draw=black, line width=0.5pt]
  (0,-0.5) -- (0,-1.5) -- (0.5,-2) -- (0.5,-1) -- cycle;
\filldraw[fill=cyan, fill opacity=0.3, draw=black, line width=0.5pt]
(1,-0.5) -- (1,-1.5) -- (0.5,-2) -- (0.5,-1) -- cycle;
%\foreach \p in {(0,-0.5),(0,-1.5),(0.5,-2),(0.5,-1),(1,-0.5),(1,-1.5)}{
%  \filldraw[black] \p circle (1.2pt);
%}
    \draw[color=black] (2.5,-2.3) node {\large{$\C A$}};
    \end{tikzpicture}
    }  
    & 
    {\vspace{45pt}
    \begin{tikzpicture}[scale=0.7,line cap=round,line join=round,x=1cm,y=1cm,baseline={(0,1)}]
    \filldraw[fill=blue, fill opacity=0.3, draw=black, line width=0.5pt]
    (-1,0.5) -- (-1,1.5) -- (0,1.5) -- (0,0.5) -- cycle;
    \filldraw[fill=orange, fill opacity=0.3, draw=black, line width=0.5pt]
    (0,0.5) -- (1,0.5) -- (1.5,1) -- (0.5,1) -- cycle;
    \filldraw[fill=orange, fill opacity=0.3, draw=black, line width=0.5pt]
    (1.5,1) -- (0.5,1) -- (0,1.5) -- (1,1.5) -- cycle;
    \filldraw[fill=orange, fill opacity=0.3, draw=black, line width=0.5pt]
    (1,0.5) -- (2,0.5) -- (2.5,1) -- (1.5,1) -- cycle;
    \filldraw[fill=orange, fill opacity=0.3, draw=black, line width=0.5pt]
    (2,0.5) -- (3,0.5) -- (3.5,1) -- (2.5,1) -- cycle;
    \filldraw[fill=cyan, fill opacity=0.3, draw=black, line width=0.5pt]
    (-1,0.5) -- (-2,0.5) -- (-2.5,1) -- (-1.5,1) -- cycle;
    \filldraw[fill=cyan, fill opacity=0.3, draw=black, line width=0.5pt]
    (-1,1.5) -- (-2,1.5) -- (-2.5,1) -- (-1.5,1) -- cycle;
    \filldraw[fill=red, fill opacity=0.4, draw=black, line width=0.5pt]
    (2.5,1) -- (2.5,2) -- (3,2.5) -- (3,1.5) -- cycle;
    \filldraw[fill=red, fill opacity=0.4, draw=black, line width=0.5pt]
    (3.5,1) -- (3.5,2) -- (3,2.5) -- (3,1.5) -- cycle;
    \draw[color=black] (0.5,-1.3) node {\large{$\C B$}};
    \end{tikzpicture}
    }    
    
\\
\hline
    {
    \begin{tikzpicture}[scale=0.7,line cap=round,line join=round,x=1cm,y=1cm,baseline={(0,-1)}]
    \filldraw[fill=blue, fill opacity=0.3, draw=black, line width=0.5pt]
    (-1,0) -- (-1,1) -- (0,1) -- (0,0) -- cycle;
    \filldraw[fill=blue, fill opacity=0.3, draw=black, line width=0.5pt]
    (1,0) -- (1,1) -- (0,1) -- (0,0) -- cycle;
    \filldraw[fill=blue, fill opacity=0.3, draw=black, line width=0.5pt]
    (1,0) -- (1,1) -- (2,1) -- (2,0) -- cycle;
    \filldraw[fill=orange, fill opacity=0.3, draw=black, line width=0.5pt]
    (2.5,0.5) -- (3.5,0.5) -- (3,0) -- (2,0) -- cycle;
    \filldraw[fill=orange, fill opacity=0.3, draw=black, line width=0.5pt]
    (2.5,0.5) -- (2,1) -- (3,1) -- (3.5,0.5) -- cycle;
    \filldraw[fill=cyan, fill opacity=0.3, draw=black, line width=0.5pt]
    (-1,0) -- (-2,0) -- (-2.5,0.5) -- (-1.5,0.5) -- cycle;
    \filldraw[fill=cyan, fill opacity=0.3, draw=black, line width=0.5pt]
    (-2.5,0.5) -- (-1.5,0.5) -- (-1,1) -- (-2,1) -- cycle;
   % \draw[color=black] (-3,-2) node {\large{$\CS{3}{2}{\C K_i}$}};
   \filldraw[fill=red, fill opacity=0.4, draw=black, line width=0.5pt]
    (1,1) -- (1,2) -- (1.5,2.5) -- (1.5,1.5) -- cycle;
    \filldraw[fill=red, fill opacity=0.4, draw=black, line width=0.5pt]
    (1.5,1.5) -- (1.5,2.5) -- (2,2) -- (2,1) -- cycle;
   \draw[color=black] (0.5,-2) node {\large{$\C C_1$}};
    \end{tikzpicture}
    }   
    & 
   {\vspace{25pt}
   \begin{tikzpicture}[scale=0.7,line cap=round,line join=round,x=1cm,y=1cm,baseline={(0,0)}]
    \filldraw[fill=blue, fill opacity=0.3, draw=black, line width=0.5pt]
    (-1,0) -- (-1,1) -- (0,1) -- (0,0) -- cycle;
    \filldraw[fill=blue, fill opacity=0.3, draw=black, line width=0.5pt]
    (1,0) -- (1,1) -- (0,1) -- (0,0) -- cycle;
    \filldraw[fill=blue, fill opacity=0.3, draw=black, line width=0.5pt]
    (1,0) -- (1,1) -- (2,1) -- (2,0) -- cycle;
    \filldraw[fill=orange, fill opacity=0.3, draw=black, line width=0.5pt]
    (2.5,0.5) -- (3.5,0.5) -- (3,0) -- (2,0) -- cycle;
    \filldraw[fill=orange, fill opacity=0.3, draw=black, line width=0.5pt]
    (2.5,0.5) -- (2,1) -- (3,1) -- (3.5,0.5) -- cycle;
    \filldraw[fill=cyan, fill opacity=0.3, draw=black, line width=0.5pt]
    (-1,0) -- (-2,0) -- (-2.5,0.5) -- (-1.5,0.5) -- cycle;
    \filldraw[fill=cyan, fill opacity=0.3, draw=black, line width=0.5pt]
    (-2.5,0.5) -- (-1.5,0.5) -- (-1,1) -- (-2,1) -- cycle;
   % \draw[color=black] (-3,-2) node {\large{$\CS{3}{2}{\C K_i}$}};
   \filldraw[fill=red, fill opacity=0.4, draw=black, line width=0.5pt]
    (1,1) -- (1,2) -- (1.5,2.5) -- (1.5,1.5) -- cycle;
    \filldraw[fill=red, fill opacity=0.4, draw=black, line width=0.5pt]
    (1.5,1.5) -- (1.5,2.5) -- (2,2) -- (2,1) -- cycle;
    \filldraw[fill=violet, fill opacity=0.4, draw=black, line width=0.5pt]
    (-1,0) -- (-1,-1) -- (-0.5,-1.5) -- (-0.5,-0.5) -- cycle;
    \filldraw[fill=violet, fill opacity=0.4, draw=black, line width=0.5pt]
    (-0.5,-1.5) -- (-0.5,-0.5) -- (0,0) -- (0,-1) -- cycle;
   \draw[color=black] (0.5,-2.5) node {\large{$\C C_2$}};
    \end{tikzpicture}
    } 

\\
\hline
    {\vspace{15pt}
    \begin{tikzpicture}[scale=0.7,line cap=round,line join=round,x=1cm,y=1cm,baseline={(0,-1)}]
    \filldraw[fill=blue, fill opacity=0.3, draw=black, line width=0.5pt]
    (-1,0) -- (-1,1) -- (0,1) -- (0,0) -- cycle;
    \filldraw[fill=cyan, fill opacity=0.3, draw=black, line width=0.5pt]
    (-1,0) -- (-2,0) -- (-2.5,0.5) -- (-1.5,0.5) -- cycle;
    \filldraw[fill=cyan, fill opacity=0.3, draw=black, line width=0.5pt]
    (-2.5,0.5) -- (-1.5,0.5) -- (-1,1) -- (-2,1) -- cycle;
    \filldraw[fill=orange, fill opacity=0.3, draw=black, line width=0.5pt]
    (0,0) -- (1,0) -- (1.5,0.5) -- (0.5,0.5) -- cycle;
    \filldraw[fill=orange, fill opacity=0.3, draw=black, line width=0.5pt]
    (1.5,0.5) -- (0.5,0.5) -- (0,1) -- (1,1) -- cycle;
    \filldraw[fill=orange, fill opacity=0.3, draw=black, line width=0.5pt]
    (0,0) -- (1,0) -- (1.5,-0.5) -- (0.5,-0.5) -- cycle;
    \filldraw[fill=red, fill opacity=0.4, draw=black, line width=0.5pt]
    (1.5,-0.5) -- (1.5,0.5) -- (2.5,0.5) -- (2.5,-0.5) -- cycle;
    \filldraw[fill=violet, fill opacity=0.4, draw=black, line width=0.5pt]
    (2.5,-0.5) -- (3.5,-0.5) -- (4,0) -- (3,0) -- cycle;
    \filldraw[fill=violet, fill opacity=0.4, draw=black, line width=0.5pt]
    (4,0) -- (3,0) -- (2.5,0.5) -- (3.5,0.5) -- cycle;
   \draw[color=black] (0.6,-2.5) node {\large{$\C E$}};
    \end{tikzpicture}
    }   
&
   {\vspace{33pt}
   \begin{tikzpicture}[scale=0.7,line cap=round,line join=round,x=1cm,y=1cm,baseline={(0,-1)}]
    \filldraw[fill=blue, fill opacity=0.3, draw=black, line width=0.5pt]
    (-1,0) -- (-1,1) -- (0,1) -- (0,0) -- cycle;
    \filldraw[fill=cyan, fill opacity=0.3, draw=black, line width=0.5pt]
    (-1,0) -- (-2,0) -- (-2.5,0.5) -- (-1.5,0.5) -- cycle;
    \filldraw[fill=cyan, fill opacity=0.3, draw=black, line width=0.5pt]
    (-2.5,0.5) -- (-1.5,0.5) -- (-1,1) -- (-2,1) -- cycle;
    \filldraw[fill=orange, fill opacity=0.3, draw=black, line width=0.5pt]
    (0,0) -- (1.5,0) -- (1,0.5) -- (0.5,0.5) -- cycle;
    \filldraw[fill=orange, fill opacity=0.3, draw=black, line width=0.5pt]
    (1,0.5) -- (0.5,0.5) -- (0,1) -- (1.5,1) -- cycle;
    \filldraw[fill=red, fill opacity=0.4, draw=black, line width=0.5pt]
    (1.5,0) -- (1.5,1) -- (2.5,1) -- (2.5,0) -- cycle;
    \filldraw[fill=red, fill opacity=0.4, draw=black, line width=0.5pt]
    (2.5,0) -- (2.5,1) -- (3.5,1) -- (3.5,0) -- cycle;
    \filldraw[fill=violet, fill opacity=0.4, draw=black, line width=0.5pt]
    (2.5,1) -- (2.5,2) -- (3,2.5) -- (3,1.5) -- cycle;
    \filldraw[fill=violet, fill opacity=0.4, draw=black, line width=0.5pt]
    (3,2.5) -- (3,1.5) -- (3.5,1) -- (3.5,2) -- cycle;
   \draw[color=black] (0.5,-2) node {\large{$\C F$}};
    \end{tikzpicture}
    } 

\\
\hline
    {\vspace{10pt}
    \begin{tikzpicture}[scale=0.7,line cap=round,line join=round,x=1cm,y=1cm,baseline={(0,-1.7)}]
   \filldraw[fill=blue, fill opacity=0.3, draw=black, line width=0.5pt]
    (-1,0) -- (-1,1) -- (0,1) -- (0,0) -- cycle;
    \filldraw[fill=cyan, fill opacity=0.3, draw=black, line width=0.5pt]
    (-1,0) -- (-2,0) -- (-2.5,0.5) -- (-1.5,0.5) -- cycle;
    \filldraw[fill=cyan, fill opacity=0.3, draw=black, line width=0.5pt]
    (-2.5,0.5) -- (-1.5,0.5) -- (-1,1) -- (-2,1) -- cycle;
    \filldraw[fill=orange, fill opacity=0.3, draw=black, line width=0.5pt]
    (0,0) -- (1,0) -- (1.5,0.5) -- (0.5,0.5) -- cycle;
    \filldraw[fill=orange, fill opacity=0.3, draw=black, line width=0.5pt]
    (1.5,0.5) -- (0.5,0.5) -- (0,1) -- (1,1) -- cycle;
    \filldraw[fill=red, fill opacity=0.4, draw=black, line width=0.5pt]
    (0,1) -- (0,2) -- (0.5,2.5) -- (0.5,1.5) -- cycle;
    \filldraw[fill=red, fill opacity=0.4, draw=black, line width=0.5pt]
    (0.5,2.5) -- (0.5,1.5) -- (1,1) -- (1,2) -- cycle;
    \filldraw[fill=violet, fill opacity=0.4, draw=black, line width=0.5pt]
    (-1,0) -- (-1,-1) -- (-1.5,-1.5) -- (-1.5,-0.5) -- cycle;
    \filldraw[fill=violet, fill opacity=0.4, draw=black, line width=0.5pt]
    (-1.5,-1.5) -- (-1.5,-0.5) -- (-2,0) -- (-2,-1) -- cycle;
    \draw[color=black] (-0.5,-2.5) node {\large{$\C S_4$}};
    \end{tikzpicture}
    }   
    & 
   {\vspace{35pt}
   \begin{tikzpicture}[scale=0.7,line cap=round,line join=round,x=1cm,y=1cm,baseline={(0,-1.5)}]
     \filldraw[fill=blue, fill opacity=0.3, draw=black, line width=0.5pt]
    (-1,0) -- (-1,1) -- (0,1) -- (0,0) -- cycle;
    \filldraw[fill=cyan, fill opacity=0.3, draw=black, line width=0.5pt]
    (-1,0) -- (-2,0) -- (-2.5,0.5) -- (-1.5,0.5) -- cycle;
    \filldraw[fill=cyan, fill opacity=0.3, draw=black, line width=0.5pt]
    (-2.5,0.5) -- (-1.5,0.5) -- (-1,1) -- (-2,1) -- cycle;
    \filldraw[fill=orange, fill opacity=0.3, draw=black, line width=0.5pt]
    (0,0) -- (1,0) -- (1.5,0.5) -- (0.5,0.5) -- cycle;
    \filldraw[fill=orange, fill opacity=0.3, draw=black, line width=0.5pt]
    (1.5,0.5) -- (0.5,0.5) -- (0,1) -- (1,1) -- cycle;
    \filldraw[fill=red, fill opacity=0.4, draw=black, line width=0.5pt]
    (0,1) -- (0,2) -- (0.5,2.5) -- (0.5,1.5) -- cycle;
    \filldraw[fill=red, fill opacity=0.4, draw=black, line width=0.5pt]
    (0.5,2.5) -- (0.5,1.5) -- (1,1) -- (1,2) -- cycle;
    \filldraw[fill=violet, fill opacity=0.4, draw=black, line width=0.5pt]
    (-1,0) -- (-1,-1) -- (-1.5,-1.5) -- (-1.5,-0.5) -- cycle;
    \filldraw[fill=violet, fill opacity=0.4, draw=black, line width=0.5pt]
    (-1.5,-1.5) -- (-1.5,-0.5) -- (-2,0) -- (-2,-1) -- cycle;
    \filldraw[fill=magenta, fill opacity=0.4, draw=black, line width=0.5pt]
    (1,0) -- (1,-1) -- (0.5,-1.5) -- (0.5,-0.5) -- cycle;
    \filldraw[fill=magenta, fill opacity=0.4, draw=black, line width=0.5pt]
    (0.5,-1.5) -- (0.5,-0.5) -- (0,0) -- (0,-1) -- cycle;
    \filldraw[fill=yellow, fill opacity=0.4, draw=black, line width=0.5pt]
    (-2,1) -- (-2,2) -- (-1.5,2.5) -- (-1.5,1.5) -- cycle;
    \filldraw[fill=yellow, fill opacity=0.4, draw=black, line width=0.5pt]
    (-1.5,2.5) -- (-1.5,1.5) -- (-1,1) -- (-1,2) -- cycle;
     \draw[color=black] (-0.5,-2.7) node {\large{$\C S_6$}};
    \end{tikzpicture}
    } 

\\
\hline

\end{tabular}
\end{scriptsize}
\caption{\label{fi:families}
Examples of $2$-clusters in $\C A$, $\C B$, $\C C_1$, $\C C_2$, $\C E$, $\C F$, $\C S_4$ and $\C S_6$ for $r=4$. Here distinct colours represent different $1$-clusters.%\op{make picture for $\mathcal F$ non-overlapping?}
}
\end{figure}

\begin{lemma}\label{lem:str}
The following two statements hold.
  \begin{itemize}
      \item For any $F\in \C M_2$ with $|F|\ge 9$, $F$ belongs to $\mathcal{A},\mathcal{B},\mathcal{C}_1,\mathcal{C}_2,\mathcal{E},\mathcal{F}$ or $\mathcal{S}_i$ with $i\in [4,3\binom{r}{2}]$.
      \item For any $F\in \C M_2$ with a composition $(e_1,\dots ,e_m)$, we have 
      \begin{equation}\label{eq:claim}
        |\p{1}{F}|=\sum _{i\in [m]}\bigg(e_i\binom{r}{2}-e_i+1\bigg)\quad\text{and}\quad |\pp{1}{2}{F}|\ge 1-m+\sum_{i=1}^m (e_i-1) (r-2)^2. 
      \end{equation}
  \end{itemize}
\end{lemma}

\begin{proof}
Consider $F\in \C M_2$ with $|F|\ge 9$ and assume that $F$ is obtained by merging $m$ 1-clusters $T_1,\dots ,T_m\in \C M_1$ in this order.

Let $s\in [m]$ be the first index such that $|\bigcup_{i=1}^{s} T_i|\geq 9$ and  let $H:=\bigcup_{i=1}^{s-1} T_i$. Then $|H|\leq 7$. As $F$ is $(8r-14,8)$-free, we have $|T_{s}|\geq 2$. 
If $T_{s}$ $1$-claims a pair which is \OutIn{1}{2}-claimed by $H$, that is, $T_s$ and $H$ are $\OMerge{}{2}$-mergeable, then by Corollary~\ref{cr:Trim} we can remove some edges from $T_{s}$ to get a $(8r-14,8)$-configuration inside $H\cup T_s$, a contradiction. Thus
$T_{s}$ must \OutIn{1}{2}-claim some pair $xy$ $1$-claimed by $H$. Let $D\subseteq T_s$ be the diamond \OutIn{1}{2}-claiming~$xy$. Note that $|H\cup D|\geq 9$, as otherwise using Corollary~\ref{cr:Trim}, we can remove some edges from $T_{s}\setminus D$ one by one to obtain an $(8r-14,8)$-configuration. Thus $|H|=7$. 

Let $T_1':=T_s$ and let $T_2'\in \{T_1,\dots,T_{s-1}\}$ be a $\Merge{}{1}$-cluster $2$-mergeable with $T_s$ via $xy$. 
Let $(T_2',\dots,T_{s}')$ be the ordering of $\{T_1,\dots,T_{s-1}\}$ returned by Lemma~\ref{lm:Trim} for the partial  $\Merge{}{2}$-clusters $T_2'\subseteq \bigcup_{i=1}^{s-1} T_i$. Therefore,  for each $i=2,\dots,s$ the \OneCluster\ $T_i'$ is $\Merge{}{2}$-mergeable with the partial \TwoCluster\ $\bigcup_{j=1}^{i-1}T_j'$. Let $t\in [s]$ be the first index such that $|\bigcup_{i=1}^{t} T_i'|\geq 9$. Set $H':=\bigcup_{i=1}^{t-1} T_i'$.
By the same argument as in the previous paragraph, we have that $|H'|=7$ and there is a diamond $D'\subseteq T_t'$ such that $H'$ and $D'$ are $\OMerge{}{2}$-mergeable.

Thus, we have a partial $\Merge{}{2}$-cluster $F':=\bigcup_{i=1}^{t} T_i'$  with at least $9$ edges built via the sequence $(T_1',T_2',\dots,T_t')$
%as in Lemma~\ref{lm:Trim} 
so that the first $\Merge{}{1}$-cluster $T_1'=T_s$ (resp.\ the last $\Merge{}{1}$-cluster $T_t'$) is merged with the rest through a pair which is \OutIn{1}{2}-claimed by the diamond $D\subseteq T_s$ (resp.\ $D'\subseteq T_t'$). Here we have the freedom to trim one or both of these two clusters, leaving any number of edges in each except exactly 1 edge. It routinely follows that $|T_1'|=|T_t'|=2$ (that is, $T_1'=D$ and $T_t'=D'$) and  the \OneCluster{}s $T_i'$ with $2\leq i\leq t-1$ contain exactly 5 edges in total.

Furthermore, we note that $F'$, with $|F'|=9$, cannot contain flexible edges, as otherwise we can remove such an edge to get an $(8r-14,8)$-configuration. This allows us to prove the following claim about the 2-cluster $F'$.

\begin{claim}\label{cl:overlap}
   Let $(S_1,S_2,\dots,S_t)$ be an arbitrary sequence of $1$-clusters that can be merged in this order to give~$F'$. 
   For any $i\in [t]$ and any flexible edge $e$ of $\bigcup_{j=1}^i S_j$, there exists a tree $S_k$ with $k\in [i+1,t]$ such that $\bigcup_{j=1}^i S_j$ $\Merge{1}{2}$-merges $S_k$ via some pair $xy$ intersecting the flexible set of $e$.
%   \begin{enumerate}[label=\rm{(\alph*)}]
%   \item For $i\in [t-2]$, we have $V(\bigcup_{j=1}^i T_j')\cap V(T'_{i+1})=\{u_i,v_i\}$;
%   \item For $i\in [2, t-1]$, we have $V(T_i')\cap V(\bigcup_{j=i+1}^t T_j')=\{u_i,v_i\}$;
 %  \item We have $V(T'_t)\cap {\bigcup_{j=2}^{t-1} T_j'}=\{u_{t-1},v_{t-1}\}$ and $V(T'_1)\cap {\bigcup_{j=2}^{t-1} T_j'}=\{u_{1},v_{1}\}$
%   \end{enumerate}
\end{claim}

\begin{proof}
    For $i\in [t]$, let $u_iv_i$ be a pair such that $\bigcup_{j=1}^{i-1} S_j$ and $S_{i}$ are $\Merge{1}{2}$-mergeable via $u_iv_i$. Due to $\cG{r}{8}$-freeness, one can easily prove by induction that for each $i\in [t-1]$, the partial 2-cluster $\bigcup_{j=1}^i S_j$ (having $\sum\limits^i_{j=1}|S_j|$ edges by definition) has exactly $(r-2)\sum\limits^i_{j=1}|S_j|+2$ vertices. Therefore the following holds:
    \begin{equation}\label{eq:uivi}
    \mbox{
    for $i\in [t-1]$, we have $V(\bigcup_{j=1}^{i-1} S_j)\cap V(S_{i})=\{u_i,v_i\}$.}
    \end{equation}
    
    Given a flexible edge $e\in \bigcup_{j=1}^i S_j$, the flexible set $W$ of $e$ must intersect some edge of $\bigcup_{j=i+1}^tS_j$, as otherwise $e$ is a flexible edge of the whole $F'$ and we can remove it to get an $8$-configuration. Let $k\in [i+1,t]$ be the smallest integer such that $V(S_k)$ and $W$ intersect. Let $x$ be a vertex in $V(S_k)\cap W$. We prove that $S_k$ satisfies the property in Claim~\ref{cl:overlap}. First, if $k\le t-1$, then by~(\ref{eq:uivi}) we have that $V(\bigcup_{j=1}^{k-1} S_j)\cap V(S_{k})=\{u_k,v_k\}$ and then $x\in \{u_k,v_k\}$. Suppose to the contrary that $S_k$ $\Merge{1}{2}$-merges with some $S_\ell$ with $i+1<\ell\le k-1$ rather than with $\bigcup_{j=1}^i S_j$. Then it means that $x\in \{u_k,v_k\}\subseteq V(S_\ell)$ and $x\in V(S_\ell)\cap W$, which contradicts to the fact that $k$ is smallest integer such that $V(S_k)$ and $W$ intersect.%
    \hide{\op{I am confused why we do not stop the proof here: it looks that we already proved all what is stated in the claim.}\sm{Since~$(\ref{eq:uivi})$ only holds for $i\le t-1$, for the final 1-cluster $S_t$, $S_t$ might overlap $\ge 3$ vertices (including $u_tv_t$) with the current structure. So far, if $k=t$, then $S_t$ might intersect $W$ via some other vertex rather than $u_t,v_t$. This does not satisfy our statement that $S_t$ $\Merge{1}{2}$-merges $\bigcup_{j=1}^i S_j$ via some pair $xy$ intersecting $W$ (namely, the merging pair $u_tv_t$ should intersect $W$)} 
    Now it suffices to assume $k=t$. 
    }
    
    Assume now that $k=t$ as otherwise we are done by the above argument.
    If there exists $x\in V(S_t)\cap W$ such that $x$ belongs to $\{u_t,v_t\}$, then by a similar argument, $S_t$ is as desired. Otherwise $\{u_t,v_t\}\cap W=\varnothing$, and $F'=\bigcup_{j=1}^{t} S_j$ has exactly $9$ edges with at most
    \[
    (r-2)\sum\limits^{t-1}_{j=1}|S_j|+2+(r-2)|S_t|-|V(S_t)\cap W|=9(r-2)+2-|V(S_t)\cap W|
    \]
    vertices. Since $k=t$ is the smallest integer such that $V(S_k)$ and $W$ intersect, if we remove the edge $e$ and the vertex set $W\setminus V(S_t)$ from $F'$ then we obtain an $r$-graph with $8$ edges and at most 
    \[
    9(r-2)+2-|V(S_t)\cap W|-|W
    \setminus V(S_t)|=9(r-2)+2-|W|=8(r-2)+2
    \]
    vertices. This contradicts the $8$-freeness of~$F'$. 
    %For $i=1$, we note that $T'_1$ is a $|T'_1|$-tree, which has exactly $|T'_1|$ edges and $(r-2)|T'_1|+2$ vertices. Now assume that $\bigcup_{j=1}^{i-1} T_j'$ has $(r-2)\sum\limits^{i-1}_{j=1}|T'_j|+2$ vertices and $\sum\limits^{i-1}_{j=1}|T'_j|$ edges. Since $F$ is $\cG{r}{8}$-free and $T'_i$ is a $|T'_i|$-tree, $T'_i$ and $\bigcup_{j=1}^{i-1} T_j'$ can overlap at most two vertices. Therefore, $\bigcup_{j=1}^i T_j'$ has $(r-2)\sum\limits^i_{j=1}|T'_j|+2$ vertices and $\sum\limits^i_{j=1}|T'_j|$ edges. By $\cG{r}{8}$-freeness, (a) holds.
%    The Items (b) and (c) can follow from similar arguments. We omit the details here.    
\end{proof}
%We can then re-order $(T_1,\cdots ,T_{\ell-1},D)$ as 
%$(T_1',\dots,T_\ell')$ so that we start with $T_1'=D$ and each new 1-cluster $T_i'$ can be merged with $\cup_{j=1}^{i-1}T_j'$ consisting of the previous ones. Then $\ell$ is the first index such that $|\bigcup_{i=1}^{\ell} T_i'|\geq 9$. Set $H':=\bigcup_{i=1}^{\ell-1} T_i'$.
%By the above argument, we have that $|H'|=7$. It follows that $T_\ell=D$ (and thus $|T_\ell|=2$), for otherwise an edge $e'\in T_\ell\setminus D$ sharing at least two vertices with $D$ would create a $(8r-14,8)$-configuration with~$H'$.

%Running the above argument on the $(T_1',\dots,T_\ell')$, we conclude that $|T_\ell'|=2$ and $\sum_{i\in [\ell]}|T_i'|=9$. Thus the 1-clusters between $T_1'$ and $T_\ell'$ have exactly 5 edges in total.

For $i\in [t]$, define  $e_i':=|T_i'|$. We know that $e_1'=e_t'=2$ and $\sum_{i=1}^te_i'=9$.
This leaves us with several possibilities for the sequence $(e_1',\dots,e_t')$.
    
\noindent\textbf{Case 1.} $(2,5,2)$

We know that the unique $5$-tree $T$ in $F'$ $\OMerge{1}{2}$-merges with two diamonds $D,D'$ via two merging pairs, say from $e,e'$, respectively. Then by $\cG{r}{8}$-freeness, we have $q(T)=2$, as otherwise we are able to remove one flexible edge not containing the merging pairs to obtain an $8$-configuration. Thus the partial 2-cluster $F'$ belongs to $\mathcal{A}$. Note that $|V(D)\cap V(D')|\le 1$, as otherwise we could trim one edge of $T$ and get an $8$-configuration. Also, this configuration $F'$ cannot $2$-merge with further 1-clusters. Indeed, if $F'$ $\OMerge{1}{2}$-merges with some 1-cluster $S_0$ via a merging pair from a diamond $D_0\subseteq S_0$, then we can trim either $D\cup \{e\}$ or $D'\cup \{e'\}$ from $F'\cup D_0$ to obtain an $8$-configuration. On the other hand, if $F'$ $\OMerge{2}{1}$-merges with some 1-cluster $S_0$ via a merging pair from an edge $e_0\in S_0$, then we can trim either $D$ or $D'$ from $F'\cup \{e_0\}$ to obtain an $8$-configuration. Therefore $F=F'$. Also,~(\ref{eq:claim}) follows easily from the above claims (combined with~\eqref{eq:uivi}).

\noindent\textbf{Case 2.} $(2,4,1,2)$ and its permutations

By Lemma~\ref{lm:Trim}, there is a merging  sequence starting with the $1$-cluster $S_1$ which is a single edge. 
Then the only edge $e$ of $S_1$ is a flexible edge, and by Claim~\ref{cl:overlap}, there exists a 1-cluster in the remaining part, say $S_2$, $\OMerge{2}{1}$-merging $S_1$. The 1-cluster $S_2$ is a $2$-tree or $4$-tree, and $e$ is still a flexible edge of $S_1\cup S_2$. Applying Lemma~\ref{lm:Trim} with the partial $\Merge{}{2}$-cluster $S_1\cup S_2$ and Claim~\ref{cl:overlap} with the flexible edge $e$ of $S_1\cup S_2$, we obtain another 1-cluster, say $S_3$, $\OMerge{2}{1}$-merging $S_1$. If $S_2$ and $S_3$ are both $2$-trees, then when $\Merge{1}{2}$-merging the remaining $4$-tree $S_4$, we can trim one edge from $S_4$ (using Corollary~\ref{cr:Trim}) to get an $8$-configuration. Thus $S_2$ and $S_3$ are a $2$-tree and a $4$-tree in some order, also they attach to $S_1$ via different pairs as otherwise $S_1\cup S_2\cup S_3$ has at least 2 flexible edges. 
Assume by symmetry that $S_3$ is a $4$-tree. Then the diamond $S_4$ must $\OMerge{2}{1}$-merge $S_3$ via a pair $u_4v_4$ intersecting a flexible set of~$S_3$ that is still flexible in $S_1\cup S_2\cup S_3$. This means that $F'\in \mathcal{B}$. 

Let us argue that no pair in $\p{12}{S_4}\setminus \{u_4v_4\}$ is $1$-claimed or $2$-claimed by $S_1\cup S_2\cup S_3$. Indeed, suppose on the contrary that there exists a pair $ab\in \p{12}{S_4}\setminus \{u_4v_4\}$ $1$-claimed or $2$-claimed by $S_1\cup S_2\cup S_3$. If $ab\in \p{12}{S_1\cup S_3}$, then $S_1\cup S_3\cup S_4$ forms a configuration of $\cG{r}{8}$. Otherwise we have $ab\in \p{12}{S_2}$ (since $\p{12}{S_1\cup S_2\cup S_3}=\p{12}{S_1\cup S_3}\cup \p{12}{S_2}$ in this case). Then we can trim $S_1$ and get an $8$-configuration $S_2\cup S_3\cup S_4$. 

Also, $F'$ cannot $2$-merge with further 1-clusters. Indeed, if $F'$ $\OMerge{1}{2}$-merges with some 1-cluster $S_0$ via a merging pair from a diamond $D_0\subseteq S_0$, then we can trim either $S_1\cup S_2$ or $S_4$ together with an edge from $S_3$ containing $u_4v_4$ from $F'\cup D_0$ to obtain an $8$-configuration. On the other hand, if $F'$ $\OMerge{2}{1}$-merges with some 1-cluster $S_0$ via a merging pair from an edge $e_0\in S_0$, then we can trim either $S_2$ or $S_4$ from $F'\cup \{e_0\}$ to obtain an $8$-configuration. 
It follows that $F=F'$ and~(\ref{eq:claim}) holds.

%In the subsequent discussions, we always denote by $u_iv_i$ the pair such that $\bigcup_{j=1}^{i-1} S_j$ and $S_{i}$ are $\Merge{1}{2}$-mergeable via $u_iv_i$. 
%\op{The pair means that it is uniquely determined but we do not know it yet for $i=t$, right? If yes then write the following:} 
In the subsequent discussions, we always denote by $u_iv_i$ a pair via which $\bigcup_{j=1}^{i-1} S_j$ and $S_{i}$ are $\Merge{1}{2}$-mergeable; note that this pair is unique if $i\le t-1$ by~\eqref{eq:uivi}.

%Note that $F'$ cannot attach or be attached to other 1-clusters. We have that $F'=F\in \mathcal{B}$ and~(\ref{eq:claim}) holds.

%We first notice that the single $1$-tree $T_1$ in the sequence should be attached at least twice, i.e. there are at least two pairs of $\partial T_1$ claimed by distinct diamonds. Indeed otherwise, $T_1$ is attached once (which is the only way for $1$-trees to be merged). Thus either $T_1$ is flexible in $F'$ with $r-2$ flexible vertices, or $T_1$ is attached once but has some other vertices lying in two diamonds $D_2,D_4$ or the $4$-tree $T_3$ in the sequence. The former case allows us to delete flexible $T_1$ and get a $(8r-14,8)$-configuration. In the latter case, we simply consider $D_2\cup T_3\cup D_4$ and obtain a $(8r-14,8)$-configuration again. 

%Hence, for two diamonds claiming $T_1$, we have two possibilities. One is that $T_1$ is attached by $D_2$ and $D_4$ while the other is that $T_1$ is attached by $D_2$ and a diamond $D$ of $T_3$. It is easy to check that the first possibility is impossible as we would always obtain a $(8r-14,8)$-configuration ($T_1\cup D_2\cup D_4$ with $3$ edges from $T_3$) no matter how $T_3$ is inserted. Moreover the last 1-cluster $D_4$ must attach to $T_3$, if not, i.e. $D_4$ attach to $T_1\cup D_2$, we would get a $(8r-14,8)$-configuration which comes from $T_1\cup D_2\cup D_4$ with $3$ edges of $T_3$. Thus $F'\in \mathcal{B}$. Note that $F'$ cannot attach or be attached to other 1-clusters. We have that $F'=F\in \mathcal{B}$ and~(\ref{eq:claim}) holds.

\noindent\textbf{Case 3.} $(2,2,3,2)$ and its permutations

Using Lemma~\ref{lm:Trim}, we consider the sequence starting from the $3$-tree $S_1$. If $S_1$ consists of three flexible edges $e_1,e_2,e_3$, i.e.\ $S_1$ is a $3$-star, then by Claim~\ref{cl:overlap}, there exists a $2$-tree $S_2$ $\Merge{1}{2}$-merging $S_1$ via some pair $xy$ where $x$ belongs to the flexible set of an edge of $S_1$, say,~$e_1$. If $S_2$ $\OMerge{1}{2}$-merges $S_1$ via $xy$, then $y$ belongs to the flexible set of $e_2$ or $e_3$. Without loss of generality, let us assume that $y\in e_2$ and $e\in S_2$ is an edge containing $xy$. Then $e_3$ and the edge of $S_2$ different from $e$ are two flexible edges of $S_1\cup S_2$. By using Claim~\ref{cl:overlap} twice, the remaining two diamonds $S_3, S_4$ will $\Merge{1}{2}$-merge with $S_1\cup S_2$ via pairs in $e_3$ and in the edge of $S_2$ different from $e$ respectively. We actually know that $S_3, S_4$ will $\OMerge{2}{1}$-merge with $S_1\cup S_2$, since otherwise we can trim one edge from $S_3$ or $S_4$ to obtain an $8$-configuration. By an argument similar to those in the previous two cases, no further $2$-merging is possible, as otherwise we can trim a diamond~(together with an edge if needed) to obtain an $8$-configuration. Also, no pair in $\p{12}{S_4}\setminus \{u_4v_4\}$ is $1$-claimed or $2$-claimed by $S_1\cup S_2\cup S_3$ by $\cG{r}{8}$-freeness. Thus $F=F'\in \C C_1$ and~(\ref{eq:claim}) holds. On the other hand, if $S_2$ $\OMerge{2}{1}$-merges $S_1$, then both $x$ and $y$ belong to $e_1$. %\op{I think $y$ can be in $e_!$ but not in its flexible set.}
We still have two flexible edges $e_2,e_3\in S_1\cup S_2$. By Claim~\ref{cl:overlap}, either $S_3$ $\OMerge{1}{2}$-merges $S_1\cup S_2$ via some pair intersecting both $e_2$ and $e_3$ (which is exactly as in the previous case), or $S_3$ $\OMerge{2}{1}$-merges $S_1\cup S_2$ via a pair in~$e_2$. In the latter case, $S_4$ will $\OMerge{2}{1}$-merge $S_1\cup S_2\cup S_3$  via a pair of $e_3$. 
By similar arguments, $F=F'\in \C C_1$ and~(\ref{eq:claim}) holds.

Now let us assume that $S_1$ has exactly two flexible edges $e_1,e_2$. Note that in this case $e_1,e_2$ cannot form a diamond. By Claim~\ref{cl:overlap}, there exist two diamonds $S_2,S_3$ $\Merge{1}{2}$-merging $S_1$ via two pairs intersecting the flexible sets of $e_1,e_2$ separately. If one of $S_2,S_3$ $\OMerge{1}{2}$-merges with $S_1$, say $S_2$, then $S_1\cup S_2\cup S_3$ would have a flexible edge  from $S_2$. Then $S_4$ will $\OMerge{2}{1}$-merge with $S_1\cup S_2\cup S_3$ (specifically, via the flexible edge of $S_2$). Let us argue that no pair in $\p{12}{S_4}\setminus \{u_4v_4\}$ is $1$-claimed or $2$-claimed by $S_1\cup S_2\cup S_3$. Indeed, suppose on the contrary that there exists a pair $ab\in \p{12}{S_4}\setminus \{u_4v_4\}$ $1$-claimed or $2$-claimed by $S_1\cup S_2\cup S_3$. If $ab\in \p{12}{S_1\cup S_2}$, then $S_1\cup S_2\cup S_4$ forms a configuration of $\cG{r}{8}$. Otherwise we have $ab\in \p{12}{S_3}$, then we can trim an edge of $S_2$ containing $u_4v_4$ from $S_1\cup S_2\cup S_3\cup S_4$ and get an $8$-configuration. Similarly to the previous cases, no further $2$-merging is possible. Again $F=F'\in \C C_1$ and~(\ref{eq:claim}) is satisfied. Suppose that both $S_2,S_3$ $\OMerge{2}{1}$-merge with $S_1$ via distinct pairs from $e_1,e_2$ respectively. Then $S_4$ also $\OMerge{2}{1}$-merges with $S_1\cup S_2\cup S_3$, as otherwise we are able to trim an edge of $S_4$ using Corollary~\ref{cr:Trim} and get an $8$-configuration. Also due to $\cG{r}{8}$-freeness,~(\ref{eq:claim}) holds. For further mergings, $F'$ cannot $\OMerge{2}{1}$-merge other \OneCluster{}s, as we could always trim $S_2$, $S_3$ or $S_4$ to find an $8$-configuration. However it is possible for $F'$ to $\OMerge{1}{2}$-merge a \OneCluster{} $S_5$. Then $S_4$ and $S_5$ must $\OMerge{2}{1}$-merge with $S_2\cup \{e_1\}$ and $S_3\cup \{e_2\}$ respectively. Indeed, if not, one can trim $S_2\cup \{e_1\}$ or $S_3\cup \{e_2\}$, obtaining an $8$-configuration. Moreover $S_5$ is a diamond, as otherwise we can trim $S_4$ and $S_2$. Similarly, by $\cG{r}{8}$-freeness, no pair in $\p{2}{S_5}\setminus \{u_5v_5\}$ is $1$-claimed or $2$-claimed by $F'$ and no further $2$-merging is possible. Hence $F=F'\cup S_5\in \C C_2$ and~(\ref{eq:claim}) holds for~$F$.

\noindent\textbf{Case 4.} $(2,1,1,3,2)$ and its permutations

By Lemma~\ref{lm:Trim}, let us re-order the sequence starting from a $1$-tree $S_1=\{e_1\}$. Notice that neither a $3$-tree nor a $1$-tree can be the last \OneCluster{} in the sequence, as otherwise we can trim one edge from $F'$ and the remaining $r$-graph is an $8$-configuration. By Claim~\ref{cl:overlap}, there exists a tree $S_2$ $\OMerge{2}{1}$-merging $S_1$. Since $e_1$ is still a flexible edge of $S_1\cup S_2$, we use Claim~\ref{cl:overlap} again and obtain a \OneCluster{} $S_3$ $\OMerge{2}{1}$-merging $S_1\cup S_2$ via a pair in $e_1$. The next \OneCluster{} $S_4$ in the sequence must be the remaining $1$-tree $\OMerge{1}{2}$-merging $S_1\cup S_2\cup S_3$. Since the $3$-tree is not the last \OneCluster{} in the sequence, one of $S_2$ and $S_3$ is the $3$-tree and the other is a diamond. Assuming $S_2$ is a diamond and $S_3$ is a $3$-tree, then there is an edge $e$ of $S_3$ such that $e$ is a flexible edge of $S_1\cup S_2\cup S_3$. Since a $1$-tree cannot be the last cluster, we derive that $S_4$ $\OMerge{1}{2}$-merges $S_1\cup S_2\cup S_3$ via a pair intersecting $e$, and the subsequent $S_5$ would $\OMerge{2}{1}$-merge with $S_4$. By $\cG{r}{8}$-freeness, no pair in $\p{2}{S_5}\setminus \{u_5v_5\}$ is $1$-claimed or $2$-claimed by $S_1\cup S_2\cup S_3\cup S_4$ and $F'$ cannot $\Merge{1}{2}$-merge more trees. Therefore $F'=F\in \mathcal{E}$ and (\ref{eq:claim}) holds.

%Let $T_1,T_2$ be two $1$-trees , $D_3,D_4$ be two diamonds and $T_5$ be the $3$-tree in the sequence. Similarly as the analysis in the Case 2, we know that $T_1$ (resp.\ $T_2$) should be attached at least twice, i.e. two of $D_3,D_4,T_5$ claims distinct shadows of $T_1$ (resp.\ $T_2$). Thus one of $D_3,D_4,T_5$ claims both $T_1$ and $T_2$, which is called {\emph{shared tree}}, and the remaining two 1-clusters claims $T_1$ and $T_2$ respectively. 

%Notice that the shared tree can only be $T_5$ (where two distinct diamonds are used), as otherwise one flexible edge in $T_5$ can be removed to get a $(8r-14,8)$-configuration. Also $F'$ cannot merge with other 1-clusters further. Therefore $F'=F\in \mathcal{E}$ and (\ref{eq:claim}) holds.

\noindent\textbf{Case 5.} $(2,2,1,2,2)$ and its permutations

Let $S_1$ be the unique $1$-tree in the sequence, and let $S_2,S_3,S_4,S_5$ be the remaining 4 diamonds. As before, by Claim~\ref{cl:overlap}, some two diamonds, say $S_2$ and $S_3$, $\OMerge{2}{1}$-merge $S_1$ via two distinct pairs of $\p{1}{S_1}$. 

If $S_4$ $\OMerge{1}{2}$-merges $S_1\cup S_2\cup S_3$ via a pair of $e\in S_4$, then the edge of $S_4$ different from $e$ is a flexible edge of $S_1\cup S_2\cup S_3\cup S_4$. The last \OneCluster{} $S_5$ would $\OMerge{2}{1}$-merge $S_4$, since a $\OMerge{1}{2}$-merging of $S_5$ and $S_4$ would enable us to trim one edge and this contradicts $8$-freeness. Similarly, due to $\cG{r}{8}$-freeness, no pair in $\p{2}{S_5}\setminus \{u_5v_5\}$ is $1$-claimed or $2$-claimed by $S_1\cup S_2\cup S_3\cup S_4$ and $F'$ admits no further $\Merge{1}{2}$-mergings. This case implies $F=F'\in \mathcal{F}$ and (\ref{eq:claim}).

The other case is that $S_4$ $\OMerge{2}{1}$-merges $S_1\cup S_2\cup S_3$. Then the next merging is that $S_5$ $\OMerge{2}{1}$-merges $S_1\cup S_2\cup S_3\cup S_4$ (as otherwise $F'$ would contain a flexible edge coming from $S_5$). Thus $F'\in \mathcal{S}_4$. It remains to prove that $F\supseteq F'$ is in $ \mathcal{S}_i$ for some $i\in [4,3\binom{r}{2}]$ and~(\ref{eq:claim}) holds. 

Let $F$ be made by starting with $F'$ and consecutively merging 1-clusters $S_6,S_7,\dots,S_m$ of $F\setminus F'$ in this order. 
%Let $\{p_1,p_2,p_3\}={e\choose 2}$ be the three  pairs inside $e$. 
Call a 1-cluster $S_i$ for $i\geq 2$ of \emph{type $ab$} if, in the merging chain from $S_1$ to the minimal partial $2$-cluster containing $S_1\cup S_i$, the first merging step is via a pair $ab\in \p{1}{S_1}$. (Note that the vertices $a,b$ are not necessarily in $S_i$, for example, $S_i$ can merge with a diamond $2$-claiming $ab$.) By convention, assume that $S_1$ is of all ${r\choose 2}$ types. 
%(Potentially, $D_i$ could be of multiple types.) 
Observe that at least two of the initial diamonds $S_2,S_3, S_4, S_5$ must be of different types (as otherwise $S_2\cup S_3\cup S_4\cup S_5$ would be an $8$-configuration).

Denote $H_i:=S_1\cup \dots\cup S_{i-1}$. To prove  that $F\in \mathcal{S}_i$ for some $i\in [4,3\binom{r}{2}]$ and~(\ref{eq:claim}) holds, it suffices to prove the following.

\begin{claim}
    For every $i\in [2,m]$, $S_i$ is a diamond that \OutIn{1}{2}-claims some previously \OutIn{}{1}-claimed pair $x_iy_i\in \p{1}{H_i}$, and no pair in $\pp{1}{2}{S_i}\setminus \{x_i,y_i\}$ is \OutIn{}{1}-claimed or \OutIn{}{2}-claimed by $H_i$. Also, $m\le 1+3\binom{r}{2}$.
\end{claim}

\begin{proof}
We prove the first part by induction on $i\in[2,m]$. It is easy to check in the base case $i\in [5]$. Let $i\geq 6$ and let the 1-cluster $S_i$ be of type~$ab$. Note that we have at most 3 diamonds of each type (otherwise the first four diamonds of any fixed type would form a forbidden $8$-configuration). If some edge $e'\in S_i$ \OutIn{}{1}-claims a pair \OutIn{1}{2}-claimed by $H_i$, then by keeping only this edge $e'$ and removing one by one diamonds of types different from $ab$, we can reach a sub-structure with exactly 8 edges, a contradiction. So let the diamond $D_i\subseteq S_i$ \OutIn{1}{2}-claim a pair $x_iy_i\in\p{1}{H_i}$. There are at most two previous diamonds $S_j, S_k$ of the same type as $S_i$ (otherwise $D_i$ with three such diamonds would give an $8$-configuration). It follows that $D_i=S_i$ as otherwise a forbidden 8-edge configuration would be formed by $S_1$, $D_i$, some suitable edge of $S_i\setminus D_i$ plus diamonds $S_j,S_k$ of type $ab$ (if exist) or one or two diamonds \OutIn{1}{2}-claiming a pair in~$\p{1}{S_1}\setminus \{ab\}$ (such diamonds exist among $S_2,S_3,S_4$). If $S_i$ contains some other vertex $z_i\not\in\{x_i,y_i\}$ from an earlier 1-cluster of the same type~$ab$, then some edge $e'$ of $S_i$ shares at least two vertices with $H_i$, again leading to a forbidden $8$-configuration in $H_i\cup \{e'\}$. Thus we are done (with proving the inductive statement and thus $F\in \mathcal{S}_i$ where $i\in [4,3\binom{r}{2}]$ and~(\ref{eq:claim}) holds for this $F$) unless $\pp{1}{2}{S_i}$ contains a pair $uv$ with both vertices in a \OneCluster{} $S_\ell$ of some different type~$a'b'$ (where $\{u,v\}$ may possibly intersect $\{x_i,y_i\}$). If the number of earlier diamonds of types $ab$ and $a'b'$ is at least 2 in total, then by trimming some diamonds if necessary, we can obtain exactly two diamonds that, together with $S_i, S_\ell$, form an $8$-configuration; otherwise we have at most 3 diamonds of type $ab$ or $a'b'$ (including $S_i, S_k$) and these diamonds together with $S_1$ have $i\leq 7$ edges and form an $(i(r-2)+1,i)$-configuration, a contradiction to $\cG{r}{8}$-freeness. 

The final inequality $m\le 1+3\binom{r}{2}$ follows from the fact that there are at most $3$ diamonds of each type $ab\in \p{1}{S_1}$.
\end{proof}

\noindent\textbf{Case 6.} $(2,2,1,1,1,2)$ and its permutations

This case is actually impossible. Indeed, let us consider the sequence starting from a $1$-tree~$S_1$. Using Claim~\ref{cl:overlap} twice, note that some two diamonds, say $S_2$ and $S_3$, would $\OMerge{2}{1}$-merge $S_1$ via two distinct pairs of $\p{1}{S_1}$. Observe that a $1$-tree cannot be the last \OneCluster{} in the sequence, as otherwise we can trim it from $F'$ and obtain  a forbidden $8$-configuration. Therefore, the next two \OneCluster{}s $S_4,S_5$ are $1$-trees $\OMerge{1}{2}$-merging with the previous partial $2$-clusters (and they cannot form a diamond by the merging rule). So $S_4,S_5$ are two flexible edges of $\cup_{i=1}^5S_i$. However, the subsequent merging of the diamond $S_6$ can eliminate only one flexible edge, say $S_4$, from $\cup_{i=1}^5S_i$, and we get a forbidden $8$-configuration by removing $S_5$ from $F'$.

%This case is actually impossible. Let $T_1,T_2,T_3$ be three $1$-trees, and let $D_4,D_5,D_6$ be three diamonds in the sequence. As before, each of $T_1,T_2,T_3$ should be attached twice (on distinct shadows) by distinct diamonds. Note that every two $1$-trees $T_j,T_k$ can have at one one shared diamond $D$, i.e. $D$ claims both $T_j,T_k$, for otherwise we obtain a forbidden $(4r-7,4)$-configuration. In other words, any two diamonds $D_j,D_k$ can only claim one $T_i$ at the same time. Without loss of generality, assume $D_4,D_5$ claim $T_1$ and $D_5,D_6$ claim $T_2$. Then $D_4\cup T_1\cup D_5\cup T_2\cup D_6$ is a $(8r-14,8)$-configuration. A contradiction.

\noindent\textbf{Case 7.} $(2,1,1,1,1,1,2)$ and its permutations

This case is also impossible. Again consider the sequence starting from a $1$-tree $S_1$. Using Claim~\ref{cl:overlap} twice, the two diamonds, say $S_2,S_3$, would $\OMerge{2}{1}$-merge $S_1$ via two distinct pairs of~$\p{1}{S_1}$. This contradicts to the fact that a $1$-tree cannot be the last \OneCluster{} in the sequence.

It remains to prove~(\ref{eq:claim}) for a 2-cluster $F$ with $|F|<9$. Here, by $8$-freeness, we have $|F|\leq 7$. When we construct $F$ by merging 1-clusters one by one, by $\cG{r}{8}$-freeness, each new 1-cluster shares exactly 2 vertices with the current configuration (namely, the pair of vertices through which the merging occurs). Thus (\ref{eq:claim}) follows. 

This finishes the proof of the Lemma~\ref{lem:str}.
\end{proof}

\subsection{Assigning weights}

%By the result from~\cite[Lemma 5]{23sh}, we assume an $r$-graph $H$ is $(8r-14,8)$-free and also $(i(r-2)+1,i)$-free, $2\le i\le 7$.  Let $\C M_1$ be the collection of 1-clusters, which is an edge partition of $H$. Each 1-cluster $X\in \C M_1$ with $j$ edges ($1\le j\le 7$) is a $j$-tree, and has $j(r-2)+2$ vertices.

%To obtain another partition $\C M_2$ of $E(H)$, we start with $\C M_1$ and iteratively merge two components if there is a pair of vertices $\{x,y\}$ such that one component covers $xy$ and the other claims $xy$. We call elements of $\C M_2$ {\emph{clusters}}.

For every 2-cluster $F$ and every pair $xy\in \binom{V(F)}{2}$, we will define some real $w_F(xy)\in [0,1]$ which we will call the \emph{weight} of $xy$ given by $F$. Then we define 
$$
w(F):=\sum_{xy\in \binom{V(F)}{2}}w_F(xy),\quad 
\mbox{for $F\in \C M_2$},$$
and 
 $$
 w(uv):=\sum_{F\in \C M_2\atop uv\in {V(F)\choose 2}}w_F(xy),\quad\mbox{for $xy\in {V(G)\choose 2}$.}
 $$
If the following inequalities hold for every $xy\in \binom{V(G)}{2}$ and $F\in \C M_2$:
\begin{eqnarray}\label{eq:one}
    w(uv)&\le& 1,
    %\quad\mbox{for every $xy\in {V(H)\choose 2}$},
    \\
\label{eq:weight}
    w(F)&\ge& \binom{r}{2}|F|,
    %\quad \mbox{for every $F\in \C M_2$},
\end{eqnarray}
then we would have that
\begin{equation}\label{eq:final_UB}
|G|=\sum _{F\in \C M_2}|F|\leq \sum_{F\in \C M_2}{\binom{r}{2}}^{-1}w(F)\leq {\binom{r}{2}}^{-1}\binom{n}{2}
\end{equation}
giving the desired upper bound. 

We shall adapt two distinct weight assignment strategies for the case when $r\ge 5$ and $r=4$.

\subsubsection{Weight functions for $r\ge 5$}\label{sec:weight-5}

For a $2$-cluster $F$ and a pair $uv\in \binom{V(F)}{2}$, we define the weight
$$w_F(uv) :=\left\{\begin{array}{ll}
1 & \text{ if } 1\in \CI_{F}(uv), \\
1/3 & \text{ if } 2\in \CI_{F}(uv) \text{ and } 1\notin \CI_{F}(uv),
\\
0 & \text{ otherwise.}
\end{array}\right.$$
The following claim shows that our weights satisfy~(\ref{eq:one}).

\begin{claim}\label{cl:w<=1}
    For every $uv\in \binom{V(G)}{2}$, it holds that $w(uv)\le 1$.
\end{claim}

\begin{proof}
Given $uv\in \binom{V(G)}{2}$, let $F_1,\dots, F_s\in \C M_2$ be all $2$-clusters assigning positive weight to $uv$, ordered so that $w_{F_1}(uv)\ge w_{F_2}(uv)\ge \dots \ge w_{F_s} (uv)$.
 %   For a fixed $uv\in \binom{V(G)}{2}$, let $F_1,\dots, F_s\in \C M_2$ be a maximal sequence of distinct $2$-clusters such that $w_{F_1}(uv)\ge w_{F_2}(uv)\ge \cdots \ge w_{F_s} (uv)>0$. 
 By the definition of weights, each $F_i$ satisfies either that $1\in \CI_{F}(uv)$ or that $ 2\in \CI_{F}(uv)$ and $1\notin \CI_{F}(uv)$. 

    If $w_{F_1}(uv)=1$, namely $1\in \CI_{F_1}(uv)$, then $2\notin \CI_F(uv)$ for any other $2$-cluster $F\in \C M_2\setminus \{F_1\}$, as otherwise we would merge them together by the merging rule of $\C M_2$. Therefore, $s=1$ and $w(uv)=w_{F_1}(uv)=1\le 1$.

    If $w_{F_1}(uv)=1/3$ (which implies $2\in \CI_{F_1}(uv)$), then by Lemma~\ref{lm: 2}, the number of $2$-clusters $F\in \C M_2\setminus \{F_1\}$ with $2\in \CI_{F}(uv)$ is at most $2$. Thus $s\le 3$ and $w(uv)\le 3\cdot w_{F_1}(uv)\le 1$.
\end{proof}

Now, it is sufficient to verify~(\ref{eq:weight}) for every $F\in \C M_2$.

%For each 2-cluster $F\in \C M_2$, we assign weight $1$ to each pair covered by $F$, i.e. $w_F(xy)=1$ for $xy\in \partial F$, and assign weight $1/3$ to each pair claimed by $F$, i.e. $w_F(xy)=\frac13$ for $xy\in c(F)$. As for other pairs $xy\notin \partial F\cup c(F)$, we assume that $w_F(xy)=0$ for now. (But we may increase these weights later without violating~(\ref{eq:one}).) At the moment,~(\ref{eq:one}) holds for all $xy\in \binom{V(H)}{2}$. Indeed, if $xy$ is covered by some $F_1\in \C M_2$, then it cannot be covered or claimed by another 2-cluster $F_2$, otherwise $F_1$ and $F_2$ would be merged. Furthermore, $xy$ cannot be claimed by four 2-clusters since $H$ is $(8r-14,8)$-free. 

%\noindent\textbf{Strategy.} Our strategy is as follows. We shall first find out 2-clusters in $\C M_2$ satisfying that
%\begin{equation}\label{eq:weight0}
 %   w_0(F)\ge {\binom{r}{2}}|F|
%\end{equation}
%where $w_0(F):=|\partial F|+\frac{1}{3}|c(F)|$. For these ``safe'' 2-clusters, we can just let $w(F)=w_0(F)$. For those 2-clusters breaking~(\ref{eq:weight0}), we then add extra weights of pairs in such 2-clusters to meet~(\ref{eq:weight}). In this procedure, we should be careful to keep~(\ref{eq:one}) always satisfied for every $xy\in \binom{V(H)}{2}$. 

%From Lemma~\ref{lem:str}, we now have a clear sight of the family $\C M_2$, especially when $|F|\ge 9$. As stated in the strategy, we begin with finding out 2-clusters $F\in \C M_2$ satisfying~(\ref{eq:weight0}). 
\begin{claim}\label{cl:ubr5}
   Let $r\ge 5$. For all $F\in \C M_2$, we have $w(F)\ge \binom{r}{2}|F|$.
\end{claim}
\begin{proof}
Suppose that $F$ is obtained by merging $m$ 1-clusters $T_1,\dots ,T_m\in \C M_1$ with $|T_i|=e_i$. Then by Lemma~\ref{lem:str}, we have 
$$|F|=\sum_{i\in [m]}e_i,\ |\p{1}{F}|=\sum _{i\in [m]}\bigg(e_i\binom{r}{2}-e_i+1\bigg)\ {\ \text{and}\ }|\pp{1}{2}{F}|\ge 1-m+\sum_{i\in [m]}(e_i-1)(r-2)^2.$$
Recall that
$$w(F):=\sum_{xy\in \binom{V(F)}{2}}w_F(xy)=|\p{1} {F}|+\frac{1}{3}\,|\pp{1}{2}{F}|.$$
By routine calculations, we have from the above that
\begin{equation}\label{eq:cal}
\begin{aligned}
2w(F)-(r^2-r)|F|&\ge (r^2-r-2)|F|+\frac{4m}{3}+\frac{2}{3}+\frac{2(r^2-4r+4)}{3}\cdot (|F|-m)-(r^2-r)|F| \\
&= \frac{2(r^2-4r+1)}{3}\,|F|+\frac{4m}{3}+\frac{2}{3}- \frac{2(r^2-4r+4)}{3}\,m\\
&= \frac{2}{3}(r^2-4r+1)\cdot(|F|-m)-\frac{2}{3}\,m+\frac{2}{3}.
\end{aligned}
\end{equation}
To see that the right hand side of (\ref{eq:cal}) is at least $0$, it suffices to verify the following:
\begin{equation}\label{eq:cal2}
    |F|\ge m+\frac{m-1}{r^2-4r+1}.
\end{equation}

%We have that $r^2-4r+1=1$ when $r=4$, and $r^2-4r+1\ge 6$ when $r\ge 5$. We shall deal with the simpler $r\ge 5$ case while $r=4$ case more complicated.

Note that $|F|\geq m+1$, since there must exist an $i$-tree in the sequence with $i\ge 2$ according to the merging rule. Hence given that $r\ge 5$, if $m\le 7$ then we derive that
$$|F|\ge m+1\ge m+\frac{m-1}{6}\ge m+\frac{m-1}{r^2-4r+1}.$$
On the other hand, if $m\geq 8$, then $|F|\ge 9$. By Lemma~\ref{lem:str}, we conclude that $F\in \mathcal{S}_{m-1}$  (as any other family has $m<8$).
%\op{We do not use $m\le 3\binom{r}{2}+1$ so I removed it} 
Thus,
$$|F|=2(m-1)+1\geq m+\frac{m-1}{r^2-4r+1}.$$

This finishes the proof of the upper bound for $r\geq 5$.
\end{proof}

\subsubsection{Weight functions for $r=4$}\label{sec:weight-4}

For a $2$-cluster $F$ and $uv\in \binom{V(G)}{2}$, we define the following functions $h_i^F$:
\begin{enumerate}[label={}]
  \item $h_1^F(uv) :=\left\{\begin{array}{ll}
1 & \text{ if } 1\in \CI_{F}(uv), \\
0 & \text{ otherwise,}
\end{array}\right.$ \hspace{17pt} $h_2^F(uv) :=\left\{\begin{array}{ll}
1/3 & \text{ if } 2\in \CI_{F}(uv), \\
0 & \text{ otherwise, }
\end{array}\right.$ 
  \item $h_3^F(uv) :=\left\{\begin{array}{ll}
1/2 & \text{ if } 2,4\in \CI_{F}(uv), \\
0 & \text{ otherwise,}
\end{array}\right.$\ \  $h_4^F(uv) :=\left\{\begin{array}{ll}
1/2 & \text{ if } 3,4,5\in \CI_{F}(uv) \text{ and } uv\notin \p{1}{G} , \\
0 & \text{ otherwise,}
\end{array}\right.$ 
  \item $h_5^F(uv) :=\left\{\begin{array}{ll}
1 & \text{ if } 3,5,6\in \CI_{F}(uv) \text{ and } uv\notin \p{1}{G}, \\
0 & \text{ otherwise.}
\end{array}\right.$
%\item $h_6(uv) :=\left\{\begin{array}{ll}
%1 & \text{ if } 3,5,6\in \CI_{F}(uv) \text{ and } uv\notin \p{1}{H}, \\
%0 & \text{ otherwise }
%\end{array}\right.$
\end{enumerate}
We assign the weight 
\[
w_{F}(uv):=\max_{1\le i\le 5}\ h_i^F(uv).
\]

\begin{claim}\label{cl:weight5}
    For every $uv\in \binom{V(G)}{2}$, it holds that $w(uv)\le 1$.
\end{claim}
\begin{proof}
Given $uv\in \binom{V(G)}{2}$, let $F_1,\dots, F_s\in \C M_2$ be all $2$-clusters assigning positive weight to $uv$, ordered so that $w_{F_1}(uv)\ge w_{F_2}(uv)\ge \dots \ge w_{F_s} (uv)$. In the following proof, let $h_i$ mean $h_i^{F_1}$, the function coming from the 2-cluster $F_1$ (that assigns the maximum weight to $uv$).

\noindent\textbf{Case 1. $w_{F_1}(uv)=1.$}

There exists some $i\in \{1,5\}$ such that $w_{F_1}(uv)=h_i(uv)=1$. 

If $w_{F_1}(uv)=h_1(uv)=1$, then $1\in \CI_{F_1}(uv)$, which means that $uv\in \p{1}{G}$. Thus, given an arbitrary $2$-cluster $F\in \C M_2\setminus \{F_1\}$, it holds that $h^F_4(uv)=h^F_5(uv)=0$.  Note that $2\notin \CI_F(uv)$, since otherwise $F_1$ would merge with $F$ by our merging rule for $\C M_2$. Thus $h^F_2(uv)=h^F_3(uv)=0$. Analogously, by the merging rule for $\C M_1$, we have $1\notin \CI_{F}(uv)$ and $h^F_1(uv)=0$. Hence, $w_{F}(uv)=0$ and $w(uv)=w_{F_1}(uv)=1$.

If $w_{F_1}(uv)=h_5(uv)=1$, then $3,5,6\in \CI_{F_1}(uv)$ and $uv\notin \p{1}{G}$. For every other $2$-cluster $F\in \C M_2\setminus \{F_1\}$,  we know by Lemma~\ref{lm: 2} that $2,3\notin \CI_{F}(uv)$ and thus $h^F_2(uv)=h^F_3(uv)=h^F_4(uv)=h^F_5(uv)=0$. Since $uv\notin \p{1}{G}$, $h^F_1(uv)=0$ as well. In total, we have $w_F(uv)=0$ and $w(uv)=w_{F_1}(uv)=1$.

\noindent\textbf{Case 2. $w_{F_1}(uv)=1/2.$}

In this case, we know $w_{F_1}(uv)\in \{h_3(uv), h_4(uv)\}$. This means $4\in \CI_{F_1}(uv)$. For every other $2$-cluster $F\in \C M_2\setminus \{F_1\}$, we derive  by Lemma~\ref{lm: 2} that $4\notin \CI_{F}(uv)$, and thus $w_{F}(uv)\le h^F_2(uv)=1/3$. Again by Lemma~\ref{lm: 2}, there is at most one $2$-cluster of $\C M_2\setminus \{F_1\}$, say $F_2$, such that $2\in \CI_{F_2}(uv)$ and $w_{F_2}(uv)=h^{F_2}_2(uv)=1/3$. Therefore, $w(uv)\le w_{F_1}(uv)+w_{F_2}(uv)\le 1/2+1/3\le 1$.

\noindent\textbf{Case 3. $w_{F_1}(uv)=1/3.$}

In this case, we know $w_{F_1}(uv)=h_2(uv)$ which implies $2\in \CI_{F_1}(uv)$. By Lemma~\ref{lm: 2}, there are at most two $2$-clusters in $\C M_2\setminus \{F_1\}$ that $2$-claim the pair $uv$. Each of these clusters (if exists) contributes at most $1/3$ while all other clusters do not contribute anything to  the weight of $uv$. Thus $w(uv)\le w_{F_1}(uv)+2/3\le 1$.
\end{proof}

%The following claim shows that our weights satisfies~(\ref{eq:one}).

\begin{claim}\label{cl:ubr4}
   Let $r=4$. For all $F\in \C M_2$, we have $w(F)\ge 6\,|F|$.
\end{claim}

\begin{proof}
Let $h_i$ refer to $h_i^F$.
Assume that $F$ is obtained by merging $m$ 1-clusters $T_1,\dots ,T_m\in \C M_1$ with $|T_i|=e_i$. Let us first focus on $1$-claimed pairs $\p{1}{F}$ and \OutIn{1}{2}-claimed pairs $\pp{1}{2}{F}$. Each pair of $\p{1}{F}$ will contribute the weight $1$ (by the definition of $h_1$), and each pair of $\pp{1}{2}{F}$ contribute at least weight $1/3$ (by the definition of $h_2$). We start with determining those $2$-clusters $F$ such that the weights contributed by $\p{1}{F}\cup \pp{1}{2}{F}$ is sufficient, i.e.\ $w(F)\ge |\p{1}{F}|+|\pp{1}{2}{F}|/3\ge 6\,|F|$.

By Lemma~\ref{lem:str}, we have 
$$|F|=\sum_{i\in [m]}e_i,\ |\p{1}{F}|=\sum _{i\in [m]}(5e_i+1)\ {\ \text{and}\ }|\pp{1}{2}{F}|\ge 1-m+4\sum_{i\in [m]}(e_i-1).$$
Then,
\begin{equation}\label{eq:exception}
\begin{aligned}
w(F)=\sum_{xy\in \binom{V(F)}{2}}w_F(xy)&\ge |\p{1} {F}|+\frac{1}{3}\,|\pp{1}{2}{F}|\\
&\ge 5\,|F|+m+\frac{1}{3}\left(1-m+4|F|-4m\right) \\
&= \frac{19}{3}\,|F|-\frac{2}{3}\,m+\frac{1}{3}.
\end{aligned}
\end{equation}
Thus if $|F|\ge 2m-1$, then~(\ref{eq:exception}) is at least $6\,|F|$.

Now it is sufficient to consider the case where $|F|<2m-1$. We split the remaining discussion depending on the value of $m$.

%We split the remaining proof depending on the value of $m$.

~\\
\noindent\textbf{Case 1.} $m\le 2.$

This case is impossible. Indeed, by the definition of $2$-merging, we have at least one $1$-cluster of size at least $2$ in the sequence of forming $F$. However, then
$|F|\geq m+1\geq 2m-1$.

\noindent\textbf{Case 2.} $m=3.$

We know $4=m+1\le |F|<2m-1=5$. Thus $|F|=4$ and $F$ has a composition $(2,1,1)$ where a diamond $D\subseteq F$ \OutIn{1}{2}-claims two  pairs $x_1y_1,x_2y_2$ through which two $1$-trees merge to $D$. These pairs are distinct as otherwise the two 1-trees would be in the same 1-cluster.
Also, note that $2,4\in \CI_{F}(x_1y_1)\cap \CI_{F}(x_2y_2)$. Thus $w_F(x_iy_i)\ge h_3(x_iy_i)=1/2$, and

$$w(F)\ge \sum_{xy\in \p{1}{F}}w_F(xy)+w_F(x_1y_1)+w_F(x_2y_2)\ge 23+\frac12+\frac12=24=6\cdot 4.$$

\noindent\textbf{Case 3.} $m=4.$

In this case, we have $5=m+1\le |F|< 2m-1=7$. We split the proof into two cases depending on the size of $F$.

First, suppose that $|F|=5$. Then $F$ possesses a composition $(2,1,1,1)$. Consider an arbitrary $1$-tree $T$ and the diamond $D$ (\OutIn{1}{2}-claiming a pair in $T$) in the sequence. It is easy to see that $|\pp{12}{3}{F}|\ge|\pp{12}{3}{D\cup T}|=8$.  By $8$-freeness, at most 2 pairs of $\pp{12}{3}{F}$ lie in $\p{1}{G}$. Therefore, the remaining pairs $uv\in \pp{12}{3}{D\cup T}\setminus \p{1}{G}$ satisfy 
$3,4,5\in \CI_{F}(uv)$ and $w_{F}(uv)\ge h_4(uv)=1/2$. Hence, by Lemma~\ref{lem:str}, we get 
$$w(F)\geq \sum_{xy\in \p{1}{F}}w_F(xy)+\sum_{xy\in \pp{12}{3}{D\cup T}\setminus \p{1}{G}}w_F(xy)\geq 29+(8-2)\times\frac12=32>6\cdot 5,$$
as desired.

Thus it remains to consider the case $|F|=6$. Then the non-increasing composition of $F$ is $(2,2,1,1)$ or $(3,1,1,1)$.

Let us show that there is at least one flexible $1$-tree $T_0$ in~$F$. Indeed, for $(2,2,1,1)$, re-order the sequence starting from a $1$-tree $e$; if $e$ is not flexible edge of $F$, then the two diamonds $\OMerge{2}{1}$-merge with $e$ and the other $1$-tree is flexible. For $(3,1,1,1)$, each $1$-tree is flexible. So a flexible 1-tree $T_0$ exists.

Let $T\neq T_0$ be another $1$-tree in $F$ (which could be flexible), and let $D\subseteq F$ be a diamond $2$-claiming a pair of $\p{1}{T}$. As before, $|\pp{12}{3}{D\cup T}|=8$. At most one pair of $|\pp{12}{3}{D\cup T}|$ lies in $\p{1}{G}$ by $\cG{4}{8}$-freeness. %\op{8-freeness (as stated before) is not enough as there can be two such pairs covered by edges in $F\setminus (D\cup T)$.}
Moreover, since $F$ has a flexible edge (namely, the one in $T_0$),  each pair $uv$ of $\pp{12}{3}{D\cup T}\setminus \p{1}{G}$ satisfies $5,6\in \CI_{F}(uv)$ and, hence, $w_{F}(uv)\ge h_5(uv)=1$. Together with Lemma~\ref{lem:str}, this implies that

$$w(F)\geq \sum_{xy\in \p{1}{F}}w_F(xy)+\sum_{xy\in \pp{12}{3}{D\cup T}\setminus \p{1}{G}}w_F(xy)\geq 34+7\times 1=41>6\cdot 6.$$

\noindent\textbf{Case 4.} $m=5.$

In this case, we have $6=m+1\le |F|< 2m-1=9$. By $8$-freeness, $|F|\neq 8$.

First suppose that $|F|=6$. Its non-increasing composition must be $(2,1,1,1,1)$. Here every $1$-tree of $F$ is flexible. Consider the diamond $D$ and a $1$-tree $T$ with $D$ $2$-claiming a pair in $T$. As before, $|\pp{12}{3}{D\cup T}|=8$ and, by $\cG{4}{8}$-freeness, there is at most one pair in $\pp{12}{3}{D\cup T}$ that lies in $\p{1}{G}$. 
%\op{8-freeness (as stated before) is not enough as there can be two such pairs covered by edges in $F\setminus (D\cup T)$.} 
By trimming one flexible edge from $F$ if needed, we know that each pair $uv\in \pp{12}{3}{D\cup T}\setminus \p{1}{G}$ satisfies $5,6\in \CI_{F}(uv)$. Hence, $w_{F}(uv)\ge h_5(uv)=1$ and by Lemma~\ref{lem:str},

$$w(F)\geq \sum_{xy\in \p{1}{F}}w_F(xy)+\sum_{xy\in \pp{12}{3}{D\cup T}\setminus \p{1}{G}}w_F(xy)\geq 35+7\times 1=42>6\cdot 6.$$

If $|F|=7$ with a composition $(2,2,1,1,1)$ or $(3,1,1,1,1)$, then we obtain at least two flexible $1$-trees $T_0,T'_0$. Indeed, for $(3,1,1,1,1)$, every $1$-tree is flexible. For $(2,2,1,1,1)$, every non-flexible $1$-tree would be $2$-claimed by both diamonds. By $\cG{4}{8}$-freeness, the number of non-flexible $1$-trees is at most $1$. Consider another $1$-tree $T\neq T_0,T'_0$ (could be flexible as well) and a diamond $D$ $2$-claiming a pair in~$T$. None of the pairs in $\pp{12}{3}{D\cup T}$ lie in $\p{1}{G}$ due to  $\cG{4}{8}$-freeness. %\op{8-freeness (as stated before) is not enough as there can be such pair covered by edges in $F\setminus (D\cup T)$.}
By removing one or two flexible edges from $F$ if needed, we know that each pair $uv\in \pp{12}{3}{D\cup T}$ satisfies $5,6\in \CI_{F}(uv)$. Thus we have $w_F(uv)\ge h_5(uv)=1$ and 

$$w(F)\geq \sum_{xy\in \p{1}{F}}w_F(xy)+\sum_{xy\in \pp{12}{3}{D\cup T}}w_F(xy)\geq 40+8\times 1=48>6\cdot 7.$$

%We terminate this case as no 2-cluster has $8$ edges.

\noindent\textbf{Case 5.} $m\geq 6.$

In this case, we have $|F|\geq m+1\geq 7$. However when $|F|= 7$, the unique composition $(2,1,1,1,1,1)$ is impossible. Recall that each \OutIn{1}{2}-claimed pair in a $4$-uniform diamond can be used for merging a $1$-tree only once, as otherwise those $1$-trees would already have been merged when constructing $\C M_1$.
Nevertheless, the number of \OutIn{1}{2}-claimed pairs in a $4$-uniform diamond is at most $4$. So a diamond cannot be merged with five $1$-trees. 

Hence $|F|\geq 9$. By Lemma~\ref{lem:str}, we know that $F\in \mathcal{S}_i$ for some $i\in [4,18]$ and $m=i+1$, as any other configuration stated in Lemma~\ref{lem:str} with more than $9$ edges satisfies $m\leq 6$. However, in this case, we have
$$|F|=2i+1=2m-1,$$
which is a contradiction to our assumption that $|F|<2m-1$.

This finishes the proof, since we considered every possible 2-cluster $F\in \C M_2$.  
\end{proof}

\subsection{Putting all together}
%\subsection{Proof of the upper bound in Theorem~\ref{thm:main}}
\begin{proof}[Proof of the upper bound in Theorem~\ref{thm:main}] 
By Lemma~\ref{lem:ub}, it is enough to prove that the size of an arbitrary $\cG{r}{8}$-free $n$-vertex $r$-graph $G$ is at most ${\binom{r}{2}}^{-1}\binom{n}{2}$. As before, let $\C M_1$ (resp. $\C M_2$) denote the partition of $E(G)$ into $1$-clusters (resp. $2$-clusters).

For each $2$-cluster $F$ and every pair $xy\in\binom{V(F)}{2}$, we define weight functions $w_F(xy)$ as in Section~\ref{sec:weight-5} (for $r\ge 5$) or Section~\ref{sec:weight-4} (for $r=4$). By Claims~\ref{cl:w<=1}--\ref{cl:ubr4}, the inequality in~(\ref{eq:final_UB}) proves the desired upper bound.
\end{proof}

%\end{proof}

\hide{
\section{Concluding remarks}

In this paper, we made progress on the Brown--Erd\H{o}s--S\'os Problem with $k=8$ and $r\ge 3$. We provided the value $\frac{1}{r^2-r}$ for $\pi(r,8)$ when $r\ge 4$ and a lower bound $3/16$ for $\pi(3,8)$. We conjecture that this lower bound is exactly the right answer.

\begin{conj}
    $\pi(3,8)=\frac{3}{16}$.
\end{conj}
}

\section*{Acknowledgements}

Both authors were supported by ERC Advanced Grant 101020255. 

The authors are grateful to the anonymous referees for many useful comments.

~\

%\noindent\textbf{Acknowledgements.} The author would like to thank Oleg Pikhurko for helpful discussions and writing suggestions.

\bibliographystyle{abbrv}
\bibliography{ref}

\end{document}